\documentclass[12pt,reqno,draft]{amsart}
\usepackage{amsmath}
\usepackage{amstext}
\usepackage{amsbsy}
\usepackage{amsopn}
\usepackage{upref}
\usepackage{amsthm}
\usepackage{amsfonts}
\usepackage{amssymb}
\usepackage{mathrsfs}
\usepackage{times}
\usepackage{fancybox} 
\usepackage{ascmac}
\allowdisplaybreaks
 \newtheorem{theorem}{Theorem}[section]

 \newtheorem{lemma}[theorem]{Lemma}
\theoremstyle{definition}
 
\theoremstyle{remark}
 \newtheorem{remark}[theorem]{Remark}
\newcommand{\ep}{\varepsilon}
\newcommand{\p}{\partial}
\newcommand{\RR}{\mathbb{R}}
\newcommand{\TT}{\mathbb{T}}

\numberwithin{equation}{section}
\begin{document}
\title[]{
Local existence of a fourth order dispersive curve flow
on locally hermitian symmetric spaces
and the application}
\author[E. ONODERA]{Eiji ONODERA}
\address[Eiji Onodera]{Department of Mathematics, 
Faculty of Science, 
Kochi University, 
Kochi 780-8520,Japan}
\email{onodera@kochi-u.ac.jp}
\begin{abstract}
This paper is concerned with a fourth order nonlinear 
dispersive partial differential equation 
for closed curve flow on a K\"ahler manifold.  
The main results is that 
the initial value problem has a solution locally in time  
if the K\"ahler manifold is a compact locally hermitian 
symmetric space. 
The proof is based on the geometric energy method 
combined with a nice gauge transformation to eliminate 
the loss of derivatives. 
Interestingly, 
the results can be applied to construct 
a generalized bi-Schr\"odinger flow 
proposed by Ding and Wang. 
The assumption on the manifold 
plays a crucial role both to enjoy a good solvable structure 
of the problem and to reduce the generalized bi-Schr\"odinger flow 
equation to the one considered in the present paper. 
\end{abstract}

\maketitle
\section{Introduction}
\label{section:introduction}
Let $N$ be a compact K\"ahler manifold  
with the complex structure $J$ 
and with a K\"ahler metric $h$, 
and let $\TT=\RR/2\pi \mathbb{Z}$ 
being the one-dimensional flat torus.  
We consider 
the initial value problem of the form 
\begin{alignat}{2}
 & u_t
  =
  a\,J_u\nabla_x^3u_x
  +
  \lambda\, J_u\nabla_xu_x
  +
  b\, R(\nabla_xu_x,u_x)J_uu_x
  +
  c\, R(J_uu_x,u_x)\nabla_xu_x
 \quad &\text{in}\quad
  &\mathbb{R}{\times} \mathbb{T},
\label{eq:pde}
\\
& u(0,x)
  =
  u_0(x)
\quad&\text{in}\quad &\mathbb{T},  
\label{eq:data}
\end{alignat}
where the solution is a map 
$u=u(t,x):\RR\times \TT\to N$  
being a closed curve flow on $N$
and the initial function is given by  
$u_0=u_0(x):\TT\to N$. 
For the map $u$, 
the velocity vector field in $t$ and $x$ 
is respectively denoted by 
$u_t=du(\frac{\p}{\p t})$ 
and 
$u_x=du(\frac{\p}{\p x})$ 
where  
$du$ is the differential of $u$, 
the covariant derivative along $u$ in $x$ is 
by $\nabla_x$, 
the complex structure at $u\in N$ is by $J_u$, 
and the Riemannian curvature tensor  
on $(N,h)$ is by $R=R(\cdot,\cdot)$. 
Moreover, 
$a$, $b$, $c$, $\lambda$ are real constants 
and $a\ne 0$ is supposed so that \eqref{eq:pde} is handled as 
a fourth order nonlinear dispersive partial differential equation. 
\par 
The equation \eqref{eq:pde} is derived by generalizing   
a two-sphere-valued physical model.  
Lakshmanan, Porsezian, and Daniel \cite{LPD} 
studied the continuum limit of the Heisenberg 
spin chain systems 
with biquadratic exchange interactions 
and formulated the model equation, 
which reads
\begin{align}
 & u_t
  =
  u\wedge
  \left[
  a_1\,\p_x^3u_{x}
  +
  \{1+a_2\, (u_x,u_x)\} \p_xu_{x}
  +
  2a_2\, (\p_xu_{x},u_x) u_{x}
  \right], 
\label{eq:pdes2}
\end{align}
where $u:\RR\times \TT\to \mathbb{S}^2\subset \RR^3$ is the solution, 
$\mathbb{S}^2$ denotes the two-dimensional unit sphere centered at the origin, 
$\p_x$ is the partial differential operator in $x$ acing on 
$\RR^3$-valued functions, 
$(\cdot,\cdot)$ and $\wedge$ denotes the inner and  
the exterior product in $\RR^3$ respectively, 
and $a_1\ne 0, a_2\in \RR$ are real physical constants. 
Moreover, \eqref{eq:pdes2} is known to arise in connection with 
the equation derived by   
Fukumoto and Moffatt \cite{fukumoto, FM} 
to model the vortex filament in an incompressible perfect fluid in $\RR^3$. 
As is found in Section~\ref{subsection:sub2},  
\eqref{eq:pdes2} is reformulated as \eqref{eq:pde}
where $N=\mathbb{S}^2$ with the complex structure 
$J_u=u\wedge$ being the $\pi/2$-degree rotation 
on $T_{u}\mathbb{S}^2$ and with  
$h(\cdot,\cdot)=(\cdot,\cdot)$ being the canonical metric 
induced from the Euclidean metric in $\RR^3$ and 
$a=a_1$, $b=5a_1-2a_2$, $c=-6a_1+3a_2$, and $\lambda=1$.   
\par 
It is to be commented that another geometric generalization 
of \eqref{eq:pdes2} 
has been proposed in \cite{onodera0}. 
The equation can be formulated by   
\begin{equation}
u_t=b_1\,J_u\nabla_x^3u_x
+b_2\,J_u\nabla_xu_x
+b_3\,h(u_x,u_x)J_u\nabla_xu_x
+b_4\,h(\nabla_xu_x,u_x)J_uu_x
\label{eq:pde9th}
\end{equation}
for $u=u(t,x):\RR\times \TT\to N$, 
where  
$(N,J,h)$ is a compact K\"ahler manifold, 
and $b_1\ne 0$, $b_2,b_3,b_4\in \RR$ are real constants. 
As in \cite{onodera0}, the equation  
\eqref{eq:pdes2} is actually generalized also as 
\eqref{eq:pde9th} where 
$N=\mathbb{S}^2$, 
$b_1=a_1$, 
$b_2=1$, 
$b_3=a_2-a_1$, 
and
$b_4=-5a_1+2a_2$. 
In fact, \eqref{eq:pde} and \eqref{eq:pde9th} 
are essentially the same, 
if $(N,J,h)$ is a Riemann surface with constant Gaussian curvature. 
See Section~\ref{subsection:sub2} for the detail. 
\par 
We are interested in the solvability of the initial value problem 
for \eqref{eq:pde}, \eqref{eq:pdes2}, and \eqref{eq:pde9th} 
in the framework of a Sobolev space. 
The main difficulty comes from the so-called loss of derivatives, 
which prevents the classical energy method from working to construct 
a time-local solution. 
More concretely, by the seemingly bad structure of first and second 
order terms in the equation, 
the estimate for the standard Sobolev norm of the solution 
can not be closed 
only by the integration by parts. 
In general, to clarify the solvability of the initial value 
problem for a dispersive 
partial differential equation, 
the structure of the lower order terms in the equation 
is to be analized in detail. 
The structure obviously becomes more complicated 
as the spatial order of the  equation becomes higher
(See, e.g., \cite{chihara2,mizuhara} for details.). 
In other words, 
the solvable structure of 
 \eqref{eq:pde}, \eqref{eq:pdes2}, and \eqref{eq:pde9th}
seems to be connected with the geometric setting of $(N,J,h)$ 
essentially.
\par 
We state here the known results in the direction. 
Guo, Zeng, and Su \cite{GZS} 
investigated the $\mathbb{S}^2$-valued model 
\eqref{eq:pdes2} and 
showed time-local existence of a weak solution to the 
initial value problem imposing that $2a_2=5a_1$. 
In the proof, two types of conservation laws 
(which do not hold if $2a_2\ne 5a_1$)
works effectively 
to overcome the difficulty of loss of derivatives. 
In recent years, for both 
\eqref{eq:pdes2} and \eqref{eq:pde9th}, 
the loss of derivatives 
have been found to be eliminated completely  
by the energy method combined with a kind of 
nice gauge transformation acting on the highest order derivatives 
of the solution. 
Indeed, the author \cite{onodera3}
studied the initial value 
problem for \eqref{eq:pdes2} 
and showed time-local existence of a smooth solution 
and the uniqueness without any assumption on $a_1\ne 0, a_2\in \RR$. 
Moreover, he showed the time-global existence  
imposing that $2a_2=5a_1$. 
After that,  
the author \cite{onodera4} 
studied the initial value problem for \eqref{eq:pde9th} and 
showed time-local existence of a smooth  
solution and the uniqueness under the assumption that 
$(N,J,h)$ is a compact Riemann surface 
with constant Gaussian curvature 
without any assumption on $b_1\ne 0,\ldots,b_4\in \RR$. 
To our interest, the necessity in some sense 
of the assumption on the curvature of 
$(N,h)$ is pointed out by Chihara \cite{chihara2} 
from the point of view of the theory of linear dispersive 
partial differential equations. 
Incidentally, 
Chihara and the author  \cite{CO2}
studied the initial value problem for \eqref{eq:pde9th} 
replacing $\TT$ with the real line $\RR$ as the spatial domain, 
and  showed time-local existence of a smooth  
solution and the uniqueness 
without any assumption on the compact K\"ahler manifold  
 $(N,J,h)$ and constants $b_1\ne 0, \ldots, b_4\in \RR$. 
The proof essentially applies the local smoothing effect 
for the fundamental solution to dispersive partial differential 
equations on $\RR$ to dominate the loss of derivatives. 
However, the method of the proof is not valid in the case of $\TT$, 
since the above dispersive smoothing effect 
is absent on compact spatial domains.  
\par 
The purpose of the present paper was 
to clarify the solvability of the initial value problem 
\eqref{eq:pde}-\eqref{eq:data} without the dispersive smoothing effect, 
which is a continuation of the work 
\cite{onodera3, onodera4}
stated above. 
The difficulty of loss of derivatives occurs also for
\eqref{eq:pde}. 
However, in view of the similarity between \eqref{eq:pde} 
and \eqref{eq:pde9th}, 
it seems natural to expect positive results 
under a suitable geometric assumption on $(N,J,h)$. 
Indeed, we have found a sufficient condition on $(N,J,h)$
to ensure the existence of a time-local smooth solution.  
More precisely, the main results  
is stated as follows: 
\begin{theorem}
\label{theorem:main}
Suppose that $(N,J,h)$  is a compact 
K\"ahler manifold satisfying $\nabla R=0$. 
Let $k$ be an integer satisfying $k\geqslant 4$. 
Then for any 
$u_0\in C(\mathbb{T};N)$ satisfying 
$u_{0x}\in H^k(\TT;TN)$, 
there exists $T=T(\|u_{0x}\|_{H^4(\TT;TN)})>0$
such that
\eqref{eq:pde}-\eqref{eq:data} 
has a solution 
$u\in C([-T,T]\times \TT;N)$
satisfying 
$u_x\in 
L^{\infty}(-T,T;H^{k}(\TT;TN))
\cap
C([-T,T];H^{k-1}(\TT;TN)).
$
\end{theorem}
\underline{Notation.} \ 
Let $\phi:\TT\to N$.   
The set of all vector fields along $\phi$
is denoted by $\Gamma(\phi^{-1}TN)$.   
The class $H^m(\TT;TN)$ for $m=0,1,2\ldots$ 
consists of all elements 
$V\in \Gamma(\phi^{-1}TN)$ 
satisfying  
$$
\|V\|_{H^m(\TT;TN)}
:=
\sum_{\ell=0}^m
\int_{\TT}
h(\nabla_x^{\ell}V(x), \nabla_x^{\ell}V(x))\,dx
<\infty.
$$
In particular, $H^0(\TT;TN)=L^2(\TT;TN)$ for simplicity. 
\vspace{0.3em}
\par
We state a contribution of Theorem~\ref{theorem:main}.  
Recall that   
\eqref{eq:pde} is nothing but \eqref{eq:pde9th} 
as far as $(N,J,h)$ is a compact Riemann surface with 
constant Gaussian curvature 
and the case has been already 
investigated in \cite{onodera3} and \cite{onodera4}. 
Hence what seems to be interesting or meaningful 
is the other cases. 
Fortunately, Theorem~\ref{theorem:main} actually 
includes such cases. 
Indeed, in Riemannian geometry, 
a K\"ahler manifold $(N,J,h)$ that satisfies $\nabla R=0$ is called 
a locally hermitian symmetric space,  
and the class of compact locally hermitian symmetric spaces is known to 
be purely wider than that of compact Riemann surfaces with constant 
Gaussian curvature. 
For example, Theorem~\ref{theorem:main} includes the case where 
$(N,J,h)$ is a compact K\"ahler manifold with constant 
holomorphic sectional curvature. 
Such a manifold is not necessarily
a Riemann surface 
with constant sectional curvature.
Incidentally, the uniqueness of the solution has been proved  
in \cite{onodera4} for \eqref{eq:pde9th} imposing the regularity 
$k\geqslant 6$. 
The proof uses an isometric embedding of the manifold into the higher
dimensional Euclidean space 
to evaluate the difference of two solutions.  
Although the author expects that the uniqueness 
holds also for \eqref{eq:pde}-\eqref{eq:data}, 
the procedure requires a lengthy computation 
and hence we do not pursue in the direction in the present paper. 
\par
To prove Theorem~\ref{theorem:main}, 
we apply the geometric energy method combined with a gauge
transformation.  
The geometric classical energy method without gauge transformations 
is known to work in the study of 
the Schr\"odinger map flow equation, 
a second order dispersive equation, 
for maps into K\"ahler manifolds. 
Indeed, under the setting, many local and global existence results 
have been established.
See, e.g., the pioneering work 
by Koiso \cite{koiso}, 
Chang, Shatah, and Uhlenbeck \cite{CSU}, 
Ding and Wang \cite{DW0}, 
Nahmod et al. \cite{NSVZ}, 
McGahagan \cite{McGahagan}, 
Rodnianski, Rubinstein, and Staffilani \cite{RRS}, 
Kenig et al. \cite{KLPST},
and references therein. 
On the other hand, the geometric energy method with a gauge transformation 
was first introduced by Chihara 
\cite{chihara}. 
He showed time-local existence and the uniqueness 
of a solution to the initial value problem for 
the Schr\"odinger map flow equation without the K\"ahler condition on the 
manifold. 
The method has been applied  
to study a fourth order dispersive flow equation 
\eqref{eq:pde9th} in \cite{CO2,onodera4} as stated above 
as well as a third order 
dispersive flow equation in \cite{CO,onodera2}. 
We slightly modify the one used in \cite{onodera4} 
to prove Theorem~\ref{theorem:main}.    
\par
More precisely, 
the idea of deciding the form 
of the gauge transformation used in the present paper is 
as follows.  
Suppose $u$ is a smooth solution 
to \eqref{eq:pde}-\eqref{eq:data}. 
If $k\geqslant 4$, the partial differential equation satisfied by 
$\nabla_x^ku_x$ 
is found to be described by  
\begin{align}
(\nabla_t-a\,J_u\nabla_x^4-d_1\,P_1\nabla_x^2-d_3\,P_2\nabla_x)
\nabla_x^ku_x
&=
\mathcal{O}
\left(
\sum_{m=0}^{k+2}
|\nabla_x^mu_x|_h
\right), 
\label{eq:esspde}
\end{align}
where $|\cdot|_h=\left\{h(\cdot,\cdot)\right\}^{1/2}$, 
$d_1$ and $d_3$ are real constants depending on 
$a,b,c,k$, 
and $P_1$ and $P_2$ are defined by 
\begin{align}
P_1Y
&=
R(Y,J_uu_x)u_x 
\quad
\text{and}
\quad
P_2Y
=
R(J_u\nabla_xu_x,u_x)Y
\nonumber
\end{align}
respectively 
for any $Y\in \Gamma(u^{-1}TN)$. 
See \eqref{eq:ayaya00} for the detail  
where $d_1$ and $d_3$ are also given exactly. 
From \eqref{eq:esspde},  
the classical energy estimate 
for $\|\nabla_x^ku_x\|_{L^2(\TT;TN)}^2$ 
is found to break down because 
the two operators $d_1\,P_1\nabla_x^2$ and $d_3\,P_2\nabla_x$
cause loss of derivatives. 
Though the right hand side of \eqref{eq:esspde} includes 
$\nabla_x^2(\nabla_x^ku_x)$ and $\nabla_x(\nabla_x^ku_x)$, 
no loss of derivatives occur thanks to 
the assumption $\nabla R=0$ and 
the K\"ahler condition $\nabla J=0$ on $(N,J,h)$.  
Hence it suffices to  eliminate the loss of derivatives 
which come from 
$d_1\,P_1\nabla_x^2$ and $d_3\,P_2\nabla_x$. 
For this purpose, 
inspired by  \cite{chihara2, onodera4}, 
we introduce the gauge transformed function $V_k$ defined by 
\begin{align}
V_k
&=
\nabla_x^ku_x
-\frac{e_1}{2a}\,
R(\nabla_x^{k-2}u_x,u_x)u_x
+
\frac{e_2}{8a}\,
R(J_uu_x,u_x)J_u\nabla_x^{k-2}u_x, 
\label{eq:igauge}
\end{align}
where 
$e_1$ and $e_2$ are constants to be decided later. 
Notice that $V_k$ is formally expressed by  
$$
V_k=\left(I_d-\frac{e_1}{2a}\Phi_1\nabla_x^{-2}
+\frac{e_2}{8a}\Phi_2\nabla_x^{-2}\right)\nabla_x^ku_x, 
$$ 
where $I_d$ is the identity on $\Gamma(u^{-1}TN)$ and 
$\Phi_1$ and $\Phi_2$ is defined by  
\begin{align}
\Phi_1Y
&=
R(Y,u_x)u_x 
\quad 
\text{and}
\quad
\Phi_2
=
R(J_uu_x,u_x)J_uY 
\nonumber
\end{align}
respectively 
for any $Y\in \Gamma(u^{-1}TN)$. 
Noting that 
$J_u$ commutes with $\Phi_2$  and not with $\Phi_1$, 
we obtain 
\begin{align}
\left[
a\,J_u\nabla_x^4, -\frac{e_1}{2a}\Phi_1\nabla_x^{-2}
\right]\nabla_x^ku_x
&=
(-e_1\,P_1\nabla_x^2-e_1\,P_2\nabla_x)\nabla_x^ku_x
+\text{harmless terms},
\label{eq:obs1}
\\
\left[
a\,J_u\nabla_x^4, \frac{e_2}{8a}\Phi_2\nabla_x^{-2}
\right]\nabla_x^ku_x
&=
-e_2\,P_2\nabla_x\nabla_x^ku_x
+
\text{harmless terms}, 
\label{eq:obs2}
\end{align}
where $[\cdot,\cdot]$ denotes the commutator bracket of two operators. 
Therefore, if we set $e_1=-d_1$ and $e_2=d_1-d_3$, 
the above two commutators eliminate 
$d_1\,P_1\nabla_x^2+d_3\,P_2\nabla_x$ in the partial 
differential equation satisfied by $V_k$, 
and hence the energy estimate for  
$\|u_x\|_{H^{k-1}(\TT;TN)}^2+\|V_k\|_{L^2(\TT;TN)}^2$ works. 
More precisely, the standard compactness argument 
with the above energy estimate for the family of 
parabolic regularized solutions  
shows the existence of a solution locally in time, 
which completes the proof.   
\par 
Let us now turn our attention to 
one more interesting conclusion  
in connection with the so-called generalized 
bi-Schr\"odinger flow. 
Recently, 
Ding and Wang \cite{DW} proposed 
a fourth order dispersive partial differential equation 
for maps into K\"ahler manifolds or para-K\"ahler manifolds, 
whose solution is called 
a generalized bi-Schr\"odinger flow. 
We recall briefly the definition and 
their work in \cite{DW} restricting to the part 
which is related to the present paper. 
Let $(M, g)$ be an $m$-dimensional Riemannian manifold 
with a Riemannian metric $g$
and let $(N,J,h)$ be a $2n$-dimensional K\"ahler manifold 
with the complex structure $J$ and a K\"ahler metric $h$.
Fix $\left\{e_1, \ldots, e_m\right\}$ 
as a local frame of $(M, g)$. 
In the local frame, $g$ is expressed by 
$g=(g_{ij})$ and its inverse is by $(g^{ij})$.  
Let $\alpha, \beta, \gamma\in \RR$ be constants where $\beta\ne 0$. 
Then the energy functional $E_{\alpha,\beta,\gamma}(u)$ 
for smooth maps
$u:(M, g)\to (N, J, h)$ is defined by  
$$
E_{\alpha,\beta,\gamma}(u)
:=
\alpha\,E(u)
+\beta\,E_2(u)
+
\gamma\,E_{\star}(u).  
$$
Here,
$E(u)
=
\frac{1}{2}
\int_M
|du|^2\,dv_g$
is the energy functional whose critical points 
are called harmonic maps 
and 
$E_2(u)
=
\frac{1}{2}
\int_M
|\tau(u)|^2\,dv_g$
is the bi-energy functional 
whose critical points are called
bi-harmonic maps, 
where  
$\tau(u)$ is the tension field along $u$
and $dv_g$ is the 
volume form of $(M,g)$. 
One can consult with \cite{ES, Jiang} for their definition. 
The energy functional $E_{\star}(u)$ is defined by 
\begin{align}
E_{\star}(u)
&=
\int_M
\sum_{i,j,k,\ell=1}^m
g^{ij}g^{kl}\,
h(R(\nabla_{e_i}u, J_u\nabla_{e_j}u)J_u\nabla_{e_k}u,
 \nabla_{e_{\ell}}u)\,dv_g, 
\nonumber
\end{align} 
where
$\nabla_{e_k}$ is the covariant derivative on the pull-back bundle 
$u^{-1}TN$ induced from the Levi-Civita connection on $(N,h)$, 
$R(\cdot,\cdot)$ is the Riemannian curvature tensor on $(N,h)$. 
The definition is of significance,
in that 
the function in the above integral 
is independent of the choice of a local frame 
$\{e_1,e_2,\ldots, e_m\}$. 
A time-dependent map 
$u=u(t,x):(-T,T)\times M\to N$  
is called a generalized bi-Schr\"odinger flow 
from $(M,g)$ to $(N,J,h)$ 
if $u$ satisfies the following Hamiltonian gradient flow equation 
\begin{equation}
u_t=J_u\nabla E_{\alpha,\beta,\gamma}(u)
\label{eq:bibi} 
\end{equation}
in $(-T,T)\times M$ for some $T>0$.  
For \eqref{eq:bibi}, 
the authors in \cite{DW} investigated 
maps from $\RR$ to 
the symmetric space $N$ where 
$N$ is one of $G_{n,k}$ or $G_n^k$. 
Here $G_{n,k}$ is the K\"ahler Grassmannian manifold of compact type 
and $G_{n}^k$  is the K\"ahler Grassmannian manifold of 
noncompact type. 
They showed that \eqref{eq:bibi} corresponds to 
the following more 
specific form 
\begin{equation}
\varphi_t=\left[
\varphi, -\alpha\,\varphi_{xx}
+\beta\,\varphi_{xxxx}
+(4\gamma-2\beta)
(\varphi_x\varphi^{-1}\varphi_x\varphi^{-1}\varphi_x)_x
\right]
\label{eq:lie}
\end{equation} 
for $\varphi:\RR\times \RR\to N$.  
Furthermore, they reduced \eqref{eq:lie} to a fourth order 
nonlinear Schr\"odinger-like matrix equation. 
See \cite{DW} for more details. 
What seems to be interesting in connection with the present paper 
is that the equation \eqref{eq:bibi} 
for maps from $\RR$ or $\TT$ 
to locally hermitian symmetric spaces
is found to be formulated as the equation \eqref{eq:pde}.   
In fact, in Section~\ref{section:biSF}, 
we formulate \eqref{eq:bibi} as \eqref{eq:high_bi_SF} 
imposing that $(M,g)$ is an $m$-dimensional flat torus 
$(\TT^m,g_0)$ and $(N,J,h)$ is a 
locally hermitian symmetric space. 
Then the  derived equation \eqref{eq:high_bi_SF} is easily found to be 
just \eqref{eq:pde} if $m=1$ (see \eqref{eq:1_bi_SF}).
To obtain the formulation, we apply the method in \cite{KLPST}. 
Though we demonstrate only the case where 
$(M,g)=(\TT^m,g_0)$, 
we can obviously obtain the  
the same equation as \eqref{eq:high_bi_SF} 
also in the case $(M,g)$ is an $m$-dimensional Euclidean space 
$(\RR^m,g_0)$. 
To summarize, 
the present paper derives another unified formulation of \eqref{eq:bibi} 
as \eqref{eq:high_bi_SF}
for maps from $(\RR^m,g_0)$ or $(\TT^m,g_0)$ 
into locally hermitian symmetric spaces
including 
$G_{n,k}$ and $G_n^k$. 
In addition, Theorem~\ref{theorem:main} 
concludes time-local existence of a generalized bi-Schr\"odinger flow 
for maps from $(\TT,g_0)$ into the compact locally hermitian symmetric 
spaces including $G_{n,k}$ and other examples which 
have not been studied so far. 
Finally, we stress that 
the assumption $\nabla R=0$ works 
both for enjoying the solvable structure of 
\eqref{eq:pde}-\eqref{eq:data} 
and for reducing \eqref{eq:bibi}  
to the form \eqref{eq:high_bi_SF}. 
\par
The organization of the present paper is as follows:
In Section~\ref{section:biSF}, 
the equation for the generalized bi-Schr\"odinger flow 
is formulated as \eqref{eq:high_bi_SF} or 
\eqref{eq:1_bi_SF}, 
and the relationship among \eqref{eq:pde}, 
\eqref{eq:pdes2}, and \eqref{eq:pde9th} 
is discussed. 
In Section~\ref{section:existence}, 
the proof of Theorem~\ref{theorem:main} is demonstrated. 
\section{The 
Generalized Bi-Schr\"odinger Flow and the Relation with Other Equations
}
\label{section:biSF}
First, we formulate  
the equation \eqref{eq:bibi} for maps 
from $(\TT^m,g_0)$ to locally hermitian symmetric spaces
as \eqref{eq:high_bi_SF} or 
\eqref{eq:1_bi_SF},
and discuss the correspondence to \eqref{eq:pde}.  
Next, we discuss the relationship among 
\eqref{eq:pde}, \eqref{eq:pdes2}, and \eqref{eq:pde9th}.
\subsection{The formulation of the generalized bi-Schr\"odinger flow.}
\label{subsection:sub1} 
In this subsection, 
suppose that $(M,g)=(\TT^m,g_0)$ and  
$(N,J,h)$ is a locally hermitian symmetric space. 
We formulate \eqref{eq:bibi}. 
\par 
Before that,
let us first recall the following basic properties 
of the Riemannian curvature tensor:  
For a map $u:\TT^m\to N$ 
and for any $Y_1, \ldots, Y_4\in \Gamma(u^{-1}TN)$, 
it holds that  
\begin{enumerate}
\item[(i)] $R(Y_1,Y_2)=-R(Y_2,Y_1)$,
\item[(ii)] $h(R(Y_1,Y_2)Y_3,Y_4)=h(R(Y_3,Y_4)Y_1,Y_2)=h(R(Y_4,Y_3)Y_2,Y_1)$,
\item[(iii)] $R(Y_1,Y_2)Y_3+R(Y_2,Y_3)Y_1+R(Y_3,Y_1)Y_2=0$,
\item[(iv)] $R(Y_1,Y_2)J_uY_3=J_u\,R(Y_1,Y_2)Y_3$, 
\item[(v)] 
$R(J_uY_1,J_uY_2)Y_3=R(Y_1,Y_2)Y_3$,
\item[(vi)]
$R(J_uY_1,Y_2)Y_3=-R(Y_1,J_uY_2)Y_3=R(J_uY_2,Y_1)Y_3$.
\end{enumerate} 
The property (iii) is called the Bianchi identity.  
The property (iv) holds since $(N,J,h)$ is a K\"ahler manifold.
The property (v) follows from (ii), (iv), and the invariance 
of $h$ under the action $J_u$.   
The property (vi) follows from (i), (v), and $J_u^2=-I_d$.  
\par 
In addition, 
the condition $\nabla R=0$ imposed on $(N,h)$ implies 
\begin{align}
&\nabla_x
\left\{
R(Y_1,Y_2)Y_3
\right\}
\nonumber
\\
&=
(\nabla_x R)(Y_1,Y_2)Y_3
+
R(\nabla_xY_1,Y_2)Y_3
+
R(Y_1,\nabla_xY_2)Y_3
+
R(Y_1,Y_2)\nabla_xY_3
\nonumber
\\
&=
R(\nabla_xY_1,Y_2)Y_3
+
R(Y_1,\nabla_xY_2)Y_3
+
R(Y_1,Y_2)\nabla_xY_3. 
\label{eq:sym}
\end{align}
These properties (i)-(vi) and \eqref{eq:sym} 
are applied not only in this section but also 
in the proof of Theorem~\ref{theorem:main}. 
\par
Let us next formulate 
three energy functionals 
stated in Introduction: 
For a map $u:\TT^m\to N$, 
$E(u)$ and $E_2(u)$  
is respectively formulated as follows: 
\begin{align}
E(u)
&=
\frac{1}{2}
\sum_{k=1}^m
\int_{\TT^m}
\left|u_{x_k}\right|_h^2\,dx
\quad
\text{and}
\quad
E_2(u)
=
\frac{1}{2}
\int_{\TT^m}
\left|
\sum_{k=1}^m
\nabla_{x_k}u_{x_k}
\right|_h^2\,dx,
\nonumber
\end{align}
where $|\cdot|_h=\left\{h(\cdot,\cdot)\right\}^{1/2}$. 
Moreover, 
since the right hand side of $E_{\star}(u)$   
is independent of the choice of 
$\left\{e_1, \ldots, e_m\right\}$, 
we set $e_k=\p/\p x_k$ ($k=1,2,\ldots,m$) to deduce 
\begin{align}
E_{\star}(u)
&=
\int_{\TT^m}
\sum_{i,j,k,\ell=1}^m
\delta_{ij}\delta_{k\ell}
h(R(\nabla_{e_i}u, J_u\nabla_{e_j}u)J_u\nabla_{e_k}u,
 \nabla_{e_{\ell}}u)
\,dx
\nonumber
\\
&=
\int_{\TT^m}
\sum_{i,k=1}^m
h(R(\nabla_{e_i}u, J_u\nabla_{e_i}u)J_u\nabla_{e_k}u,
 \nabla_{e_{k}}u)
\,dx
\nonumber
\\
&=
\int_{\TT^m}
\sum_{i,k=1}^m
h(R(u_{x_i}, J_uu_{x_i})
J_uu_{x_k},
u_{x_k})
\,dx.
\nonumber
\end{align} 
 \par 
Having them in mind,  
we are now ready to compute the gradient flow 
$$
\nabla E_{\alpha,\beta,\gamma}(u)
=
\alpha\,\nabla E(u)
+
\beta\,\nabla E_{2}(u)
+
\gamma\,\nabla E_{\star}(u).
$$  
We demonstrate the computation of 
$\nabla E_{\star}(u)$ below. 
First,  
following the method in \cite{KLPST},  
we construct the variation of $u:\TT^m\to N$ 
with given initial velocity 
$\xi\in \Gamma(u^{-1}TN)$ 
by    
$U:\TT^m\times \RR\to N$, 
where 
$U(x,\ep)=\exp_{u(x)}\left[ \ep\,\xi(x)\right]$ 
and 
$\exp_{u(x)}:T_{u(x)}N\to N$ 
is the exponential map at $u(x)\in N$. 
Next, since  
$\nabla E_{\star}(u)$ 
is given by 
\begin{equation}
\label{eq:grad}
\frac{d}{d \ep}
E_{\star}(U)\biggl|_{\ep=0}
=
\int_{\TT^m}
h(\nabla E_{\star}(u), \xi)\,dx, 
\end{equation}
we compute the left hand side of \eqref{eq:grad}, 
where $\nabla R=0$ is used. 
Indeed, by using (ii) and \eqref{eq:sym} with $u=U$, 
we see 
\begin{align}
\frac{d}{d \ep}
E_{\star}(U)
&=
\sum_{i,k=1}^m
\int_{\TT^m}
h\left(
R(\nabla_{\ep}U_{x_i}, J_UU_{x_i})J_UU_{x_k}, 
U_{x_k}
\right)\,
dx
\nonumber
\\
&\quad
+
\sum_{i,k=1}^m
\int_{\TT^m}
h\left(
R(U_{x_i}, \nabla_{\ep}J_UU_{x_i})J_UU_{x_k}, 
U_{x_k}
\right)\,
dx
\nonumber
\\
&\quad
+
\sum_{i,k=1}^m
\int_{\TT^m}
h\left(
R(U_{x_i}, J_UU_{x_i})\nabla_{\ep}J_UU_{x_k}, 
U_{x_k}
\right)\,
dx
\nonumber
\\
&\quad
+
\sum_{i,k=1}^m
\int_{\TT^m}
h\left(
R(U_{x_i}, J_UU_{x_i})J_UU_{x_k}, 
\nabla_{\ep}U_{x_k}
\right)\,
dx
\nonumber
\\
&=
2
\sum_{i,k=1}^m
\int_{\TT^m}
h\left(
R(\nabla_{\ep}U_{x_i}, J_UU_{x_i})J_UU_{x_k}, 
U_{x_k}
\right)\,
dx
\nonumber
\\
&\quad
+
2
\sum_{i,k=1}^m
\int_{\TT^m}
h\left(
R(U_{x_i}, \nabla_{\ep}J_UU_{x_i})J_UU_{x_k}, 
U_{x_k}
\right)\,
dx. 
\nonumber
\end{align}
Since $(N,J,h)$ is a K\"ahler manifold, 
$\nabla_{\ep}J_U=J_U\nabla_{\ep}$ holds. 
Using this and (vi) with $u=U$, we have
\begin{align}
h\left(
R(U_{x_i}, \nabla_{\ep}J_UU_{x_i})J_UU_{x_k}, 
U_{x_k}
\right)
&=
h\left(
R(U_{x_i}, J_U\nabla_{\ep}U_{x_i})J_UU_{x_k}, 
U_{x_k}
\right)
\nonumber
\\
&=
h\left(
R(\nabla_{\ep}U_{x_i}, J_U U_{x_i})J_UU_{x_k}, 
U_{x_k}
\right), 
\nonumber
\end{align}
which implies 
\begin{align}
\frac{d}{d \ep}
E_{\star}(U)
&=
4
\sum_{i,k=1}^m
\int_{\TT^m}
h\left(
R(\nabla_{\ep}U_{x_i}, J_UU_{x_i})J_UU_{x_k}, 
U_{x_k}
\right)\,
dx.
\nonumber
\end{align}
Furthermore, by integrating by parts and by 
using $\nabla_{\ep}U_{x_i}=\nabla_{x_i}U_{\ep}$, 
(i), (ii) and (vi) with $u=U$, 
we obtain 
\begin{align}
\frac{d}{d \ep}
E_{\star}(U)
&=
4
\sum_{i,k=1}^m
\int_{\TT^m}
h\left(
R(\nabla_{x_i}U_{\ep}, J_UU_{x_i})J_UU_{x_k}, 
U_{x_k}
\right)\,
dx
\nonumber
\\
&=
-4
\sum_{i,k=1}^m
\int_{\TT^m}
h\left(
R(U_{\ep}, J_U\nabla_{x_i}U_{x_i})J_UU_{x_k}, 
U_{x_k}
\right)\,
dx
\nonumber
\\
&\quad
-4
\sum_{i,k=1}^m
\int_{\TT^m}
h\left(
R(U_{\ep}, J_UU_{x_i})J_U\nabla_{x_i}U_{x_k}, 
U_{x_k}
\right)\,
dx
\nonumber
\\
&\quad
-4
\sum_{i,k=1}^m
\int_{\TT^m}
h\left(
R(U_{\ep}, J_UU_{x_i})J_UU_{x_k}, 
\nabla_{x_i}U_{x_k}
\right)\,
dx
\nonumber
\\
&=
-4
\sum_{i,k=1}^m
\int_{\TT^m}
h\left(
R(U_{\ep}, J_U\nabla_{x_i}U_{x_i})J_UU_{x_k}, 
U_{x_k}
\right)\,
dx
\nonumber
\\
&\quad
+8
\sum_{i,k=1}^m
\int_{\TT^m}
h\left(
R(U_{\ep}, J_UU_{x_i})\nabla_{x_i}U_{x_k}, 
J_UU_{x_k}
\right)\,
dx.
\nonumber
\end{align}
Note that 
$U_{\ep}(x,\ep)|_{\ep =0}=\xi(x)$
and $U(x,\ep)|_{\ep=0}=u(x)$ 
follow from the definition of the map $U$. 
Therefore, by letting $\ep\to 0$ and by using (ii), we get 
\begin{align}
\frac{d}{d \ep}
E_{\star}(U)
\biggl|_{\ep=0}
&=
-4
\sum_{i,k=1}^m
\int_{\TT^m}
h\left(
R(\xi, J_u\nabla_{x_i}u_{x_i})J_uu_{x_k}, 
u_{x_k}
\right)\,
dx
\nonumber
\\
&\quad
+8
\sum_{i,k=1}^m
\int_{\TT^m}
h\left(
R(\xi, J_uu_{x_i})\nabla_{x_i}u_{x_k}, 
J_uu_{x_k}
\right)\,
dx
\nonumber
\\
&=
-4
\sum_{i,k=1}^m
\int_{\TT^m}
h\left(
\xi, 
R(u_{x_k}, J_uu_{x_k})J_u\nabla_{x_i}u_{x_i}
\right)\,
dx
\nonumber
\\
&\quad
+8
\sum_{i,k=1}^m
\int_{\TT^m}
h\left(
\xi, 
R(J_uu_{x_k}, \nabla_{x_i}u_{x_k})J_uu_{x_i}
\right)\,
dx.
\nonumber
\end{align}
Moreover, noting $h$ is invariant under $J_u$ and $J_u^2=-I_d$, 
we use (iii) and (iv)  
to deduce  
\begin{align}
&h\left(
\xi, 
R(J_uu_{x_k}, \nabla_{x_i}u_{x_k})J_uu_{x_i}
\right)
\nonumber
\\
&=
-h\left(
J_u\xi, 
R(J_uu_{x_k}, \nabla_{x_i}u_{x_k})u_{x_i}
\right)
\quad
(\because \text{(iv)})
\nonumber
\\
&=
h\left(
J_u\xi, R( \nabla_{x_i}u_{x_k}, u_{x_i})J_uu_{x_k}
\right)
+
h\left(
J_u\xi, R(u_{x_i},J_uu_{x_k})\nabla_{x_i}u_{x_k}
\right)
\quad (\because \text{(iii)})
\nonumber
\\
&=
h\left(
\xi, R( \nabla_{x_i}u_{x_k}, u_{x_i})u_{x_k}
\right)
-h\left(
\xi, R(u_{x_i},J_uu_{x_k})J_u\nabla_{x_i}u_{x_k}
\right).
\quad
(\because \text{(iv)})
\nonumber
\end{align}
Substituting this, we obtain 
\begin{align}
\frac{d}{d \ep}
E_{\star}(U)
\biggl|_{\ep=0}
&=
-4
\sum_{i,k=1}^m
\int_{\TT^m}
h\left(
\xi, 
R(u_{x_k}, J_uu_{x_k})J_u\nabla_{x_i}u_{x_i}
\right)\,
dx
\nonumber
\\
&\quad
+8
\sum_{i,k=1}^m
\int_{\TT^m}
h\left(
\xi, R( \nabla_{x_i}u_{x_k}, u_{x_i})u_{x_k}
\right)\,dx
\nonumber
\\
&\quad 
-8
\sum_{i,k=1}^m
\int_{\TT^m}
h\left(
\xi, R(u_{x_i},J_uu_{x_k})J_u\nabla_{x_i}u_{x_k}
\right)
\,dx.
\nonumber
\end{align}
Thus, comparing with \eqref{eq:grad}, we obtain  
\begin{align}
\nabla E_{\star}(u)
&=
-4
\sum_{i,k=1}^m
R(u_{x_k}, J_uu_{x_k})J_u\nabla_{x_i}u_{x_i}
+
8
\sum_{i,k=1}^m
R( \nabla_{x_i}u_{x_k}, u_{x_i})u_{x_k}
\nonumber
\\
&\quad 
-8
\sum_{i,k=1}^m
R(u_{x_i},J_uu_{x_k})J_u\nabla_{x_i}u_{x_k}. 
\label{eq:IIIIII}
\end{align} 
In the same way, we can compute to see  
\begin{align}
\nabla E(u)
&=-\sum_{k=1}^m\nabla_{x_k}u_{x_k},
\label{eq:IIII}
\\
\nabla E_2(u)
&=
\sum_{k,\ell=1}^m
\left\{
\nabla_{x_k}^2
\nabla_{x_{\ell}}u_{x_{\ell}}
+
R(\nabla_{x_{\ell}}u_{x_{\ell}}, u_{x_k})u_{x_k}
\right\}.
\label{eq:IIIII}
\end{align}
We omit the detail, 
since \eqref{eq:IIII} and \eqref{eq:IIIII} 
are already well-known. 
(See, e.g., \cite[Section~2.2]{KLPST} for \eqref{eq:IIIII}.)
Indeed, 
the form of the right hand side of 
\eqref{eq:IIII} and \eqref{eq:IIIII}
actually agrees with that of 
the harmonic map equation and the bi-harmonic map equation 
respectively.  
\par
Combining \eqref{eq:IIIIII}, \eqref{eq:IIII}, and 
\eqref{eq:IIIII}, 
we obtain 
\begin{align}
\nabla E_{\alpha,\beta,\gamma}(u)
&=
\beta\,\sum_{k,\ell=1}^m
\nabla_{x_k}^2
\nabla_{x_{\ell}}u_{x_{\ell}}
-\alpha\,\sum_{k=1}^m\nabla_{x_k}u_{x_k}
\nonumber
\\
&\quad 
+
\beta\,\sum_{k,\ell=1}^m
R(\nabla_{x_{\ell}}u_{x_{\ell}}, u_{x_k})u_{x_k}
-4\gamma\,
\sum_{i,k=1}^m
R(u_{x_k}, J_uu_{x_k})J_u\nabla_{x_i}u_{x_i}
\nonumber
\\
&\quad 
+8\gamma\,
\sum_{i,k=1}^m
R( \nabla_{x_i}u_{x_k}, u_{x_i})u_{x_k}
-8\gamma\,
\sum_{i,k=1}^m
R(u_{x_i},J_uu_{x_k})J_u\nabla_{x_i}u_{x_k}. 
\nonumber 
\end{align}
Therefore,  
by using (iv) and (vi), 
we see that the partial differential equation 
for the generalized bi-Schr\"odinger flow is 
governed by 
\begin{align}
u_t
&=
\beta\,J_u\sum_{k,\ell=1}^m
\nabla_{x_k}^2
\nabla_{x_{\ell}}u_{x_{\ell}}
-\alpha\,J_u\sum_{k=1}^m\nabla_{x_k}u_{x_k}
\nonumber
\\
&\quad 
+
\beta\,\sum_{k,\ell=1}^m
R(\nabla_{x_{\ell}}u_{x_{\ell}}, u_{x_k})J_uu_{x_k} 
+4\gamma\,
\sum_{i,k=1}^m
R(u_{x_k}, J_uu_{x_k})\nabla_{x_i}u_{x_i}
\nonumber
\\
&\quad 
+8\gamma\,
\sum_{i,k=1}^m
R( \nabla_{x_i}u_{x_k}, u_{x_i})J_uu_{x_k}
+8\gamma\,
\sum_{i,k=1}^m
R(u_{x_i},J_uu_{x_k})\nabla_{x_i}u_{x_k}
\nonumber
\\
&=
\beta\,J_u\sum_{k,\ell=1}^m
\nabla_{x_k}^2
\nabla_{x_{\ell}}u_{x_{\ell}}
-\alpha\,J_u\sum_{k=1}^m\nabla_{x_k}u_{x_k}
\nonumber
\\
&\quad 
+
\beta\,\sum_{k,\ell=1}^m
R(\nabla_{x_{\ell}}u_{x_{\ell}}, u_{x_k})J_uu_{x_k} 
-4\gamma\,
\sum_{i,k=1}^m
R(J_uu_{x_k}, u_{x_k})\nabla_{x_i}u_{x_i}
\nonumber
\\
&\quad 
+8\gamma\,
\sum_{i,k=1}^m
R( \nabla_{x_i}u_{x_k}, u_{x_i})J_uu_{x_k}
-8\gamma\,
\sum_{i,k=1}^m
R(J_uu_{x_i},u_{x_k})\nabla_{x_i}u_{x_k}.
\label{eq:high_bi_SF}
\end{align}
\par 
Specifically if $m=1$, 
it is obvious that
the equation \eqref{eq:high_bi_SF} 
reads 
\begin{align}
u_t
&=
\beta\,J_u\nabla_x^3u_x
-\alpha\,J_u\nabla_{x}u_{x}
\nonumber
\\
&\quad 
+
(\beta+8\gamma)\,
R(\nabla_{x}u_{x}, u_{x})J_uu_{x} 
-12\gamma\,
R(J_uu_{x}, u_x)\nabla_{x}u_{x}, 
\label{eq:1_bi_SF}
\end{align}
which is just \eqref{eq:pde} 
where 
$a=\beta$, $\lambda=-\alpha$, 
$b=\beta+8\gamma$, 
$c=-12\gamma$. 
\subsection{The relationship among \eqref{eq:pde}, 
\eqref{eq:pdes2}, and \eqref{eq:pde9th}}
\label{subsection:sub2}
Suppose that $(N,J,h)$ is a Riemann surface 
with constant sectional curvature $S$. 
Let $u$ be a smooth solution to \eqref{eq:pde}. 
By the definition of the sectional curvature, 
$$
R(X,Y)Z
=S
\left\{
h(Y,Z)X-h(X,Z)Y
\right\}
$$
holds for any $X,Y,Z\in \Gamma(u^{-1}TN)$. 
Using this, (iv), and $J_u^2=-I_d$,  
we have 
\begin{align}
R(\nabla_xu_x,u_x)J_uu_x
&= 
J_uR(\nabla_xu_x,u_x)u_x
\nonumber
\\
&=
S\,h(u_x,u_x)J_u\nabla_xu_x
-
S\,h(\nabla_xu_x,u_x)J_uu_x, 
\label{eq:dore1}
\\
R(J_uu_x,u_x)\nabla_xu_x
&=
S\,\left\{
h(u_x,\nabla_xu_x)J_uu_x
-h(J_uu_x, \nabla_xu_x)u_x
\right\}
\nonumber
\\
&=
S\,
h(\nabla_xu_x,u_x)J_uu_x
+
S\,J_uh(\nabla_xu_x, J_uu_x)J_uu_x. 
\label{eq:dore2}
\end{align} 
Since $N$ is here a two-dimensional real manifold, 
$\dfrac{u_x(t,x)}{|u_x(t,x)|_h}$ 
and $\dfrac{J_{u(t,x)}u_x(t,x)}{|u_x(t,x)|_h}$ 
form a basis for 
$T_{u(t,x)}N$ if $u_x(t,x)\ne 0$. 
Therefore, we notice that 
\begin{equation}
h(u_x,u_x)Y=
h(Y,u_x)u_x
+
h(Y,J_uu_x)J_uu_x
\label{eq:frame}
\end{equation}
holds for any $Y\in \Gamma(u^{-1}TN)$. 
Using \eqref{eq:frame} with $Y=\nabla_xu_x$,  
we have 
\begin{align}
J_u\,h(\nabla_xu_x,J_uu_x)J_uu_x
&=
J_u\left\{
h(u_x,u_x)\nabla_xu_x
-h(\nabla_xu_x,u_x)u_x
\right\}
\nonumber
\\
&=
h(u_x,u_x)J_u\nabla_xu_x
-
h(\nabla_xu_x,u_x)J_uu_x. 
\label{eq:dore3}
\end{align}
Substituting \eqref{eq:dore3} into 
\eqref{eq:dore2}, 
we have 
\begin{equation}
R(J_uu_x,u_x)\nabla_xu_x
=
S\,h(u_x,u_x)J_u\nabla_xu_x.
\label{eq:dore4}
\end{equation} 
Collecting \eqref{eq:dore1} and \eqref{eq:dore4}, 
we see that the equation \eqref{eq:pde} reads 
\begin{equation}
u_t
=a\,J_u\nabla_x^3u_x
+\lambda\,J_u\nabla_xu_x
+(b+c)S\,h(u_x,u_x)J_u\nabla_xu_x
-bS\,h(\nabla_xu_x,u_x)J_uu_x,  
\nonumber
\end{equation}
which is nothing but \eqref{eq:pde9th} 
where $b_1=a$, $b_2=\lambda$, 
$b_3=(b+c)S$, and $b_4=-bS$.  
\par 
Furthermore, it follows from \cite{onodera0} that
the $\mathbb{S}^2$-valued model 
\eqref{eq:pdes2} 
is reformulated as 
\eqref{eq:pde9th}  
where $N$ is the canonical two-sphere $\mathbb{S}^2$ 
with $S=1$
and 
$b_1=a_1$, $b_2=1$, $b_3=a_2-a_1$, $b_4=-5a_1+2a_2$. 
In other words, \eqref{eq:pdes2} is reformulated as
\eqref{eq:pde} where $N=\mathbb{S}^2$, 
$a=a_1$, $\lambda=1$, $b=5a_1-2a_2$, $c=-6a_1+3a_2$.
\section{Local Existence}
\label{section:existence}
The purpose of this section is to prove Theorem~\ref{theorem:main}.
\begin{proof}[Proof of Theorem~\ref{theorem:main}]
Suppose $k\geqslant 4$. 
It suffices to construct a solution in the positive time direction. 
We start with the case where $u_{0}\in C^{\infty}(\TT;N)$. 
We can construct a family of parabolic regularized 
approximating solutions by solving   
\begin{alignat}{2}
 & u_t
  =
 (-\ep + a\,J_u)\nabla_x^3u_x
  \nonumber 
  \\
  &\quad \quad
  +
  b\, R(\nabla_xu_x,u_x)J_uu_x
  +
  c\,R(J_uu_x,u_x)\nabla_xu_x
  +\lambda\, J_u\nabla_xu_x
 \quad &\text{in}\quad
 &(0,\infty){\times} \mathbb{T},
\label{eq:eppde}
\\
& u(0,x)
  =
  u_0(x)
\quad&\text{in}\quad &\mathbb{T}
\label{eq:epdata}
\end{alignat} 
for each fixed $\ep\in (0,1]$. 
Since the parabolic term 
$-\ep\,\nabla_x^3u_x$ is added, 
we can handle \eqref{eq:eppde} as a fourth-order 
quasilinear parabolic system. 
In fact, we can show the following:    
\begin{lemma}
\label{lemma:parabolic}  
For each $\ep\in (0,1]$, 
there exists a positive constant $T_{\ep}$ 
depending on $\ep$ and $\|u_{0x}\|_{H^4(\TT;TN)}$ 
such that  
\eqref{eq:eppde}-\eqref{eq:epdata} possesses a 
unique solution $u\in C^{\infty}([0,T_{\ep}]\times \TT;N)$. 
\end{lemma}
Lemma~\ref{lemma:parabolic} 
can be easily proved 
by following the argument in 
\cite[Lemma~3.1]{CO2} or \cite[Lemma~2.2]{onodera4}.  
The argument is based on the mix of a sixth-order 
parabolic regularization and 
geometric classical energy method. 
In what follows, 
we denote by $u^{\ep}$ the solution to 
\eqref{eq:eppde}-\eqref{eq:epdata} in 
Lemma~\ref{lemma:parabolic}. 
This presents the family $\left\{u^{\ep}\right\}_{\ep\in(0,1]}$. 
\par  
We next obtain a uniform lower bound 
$T$ of $\left\{T_{\ep}\right\}_{\ep\in (0,1]}$ 
and show that $\left\{u^{\ep}_x\right\}_{\ep\in (0,1]}$ 
is bounded in $L^{\infty}(0,T;H^k(\TT;TN))$. 
The classical 
energy estimate for  $\|u^{\ep}_x\|_{H^k(\TT;TN)}$  
is found to cause loss of derivatives. 
To eliminate the loss of derivatives,   
we introduce a gauge transformed function $V^{\ep}_k$ 
defined by   
\begin{align}
V^{\ep}_k
&=
\nabla_x^ku_x^{\ep}
+
\Lambda^{\ep}, 
\label{eq:V_m}
\end{align}
where $\Lambda^{\ep}
=
\Lambda^{\ep}_1
+
\Lambda^{\ep}_2$ 
with 
\begin{align}
\Lambda_1^{\ep}
&:=
-\frac{e_1}{2a}\,
R(\nabla_x^{k-2}u_x^{\ep},u_x^{\ep})u_x^{\ep}
=-\frac{e_1}{2a}\Phi_1\nabla_x^{-2}(\nabla_x^ku_x^{\ep}), 
\nonumber
\\
\Lambda_2^{\ep}
&:=
\frac{e_2}{8a}\,
R(J_{u^{\ep}}u_x^{\ep},u_x^{\ep})J_{u^{\ep}}\nabla_x^{k-2}u_x^{\ep}
=\frac{e_2}{8a}\Phi_2\nabla_x^{-2}(\nabla_x^ku_x^{\ep}).  
\nonumber
\end{align}
Here $\Phi_1$ and $\Phi_2$ are defined in Introduction 
and $e_1, e_2\in\RR$ are real constants to be decided later. 
Instead of the energy estimate for  $\|u^{\ep}_x\|_{H^k(\TT;TN)}$, 
we consider the estimate for the energy 
$N_k(u^{\ep})$ defined by  
\begin{equation}
N_k(u^{\ep}(t))
=
\sqrt{
\|u_x^{\ep}(t)\|_{H^{k-1}(\TT;TN)}^2
+
\|V_k^{\ep}(t)\|_{L^2(\TT;TN)}^2
}
\label{eq:N_k}
\end{equation}
for $t\in [0,T^{\ep}]$. 
Before the estimate, 
we restrict the time interval on $[0,T_{\ep}^{\star}]$ 
with $T^{\star}_{\ep}$ defined by 
$$
T^{\star}_{\ep}
=
\sup
\left\{
T>0 \ | \ 
N_4(u^{\ep}(t))\leqslant 2N_4(u_0)
\quad 
\text{for all}
\quad
t\in[0,T]
\right\}. 
$$
The restriction of the time-interval and the Sobolev embedding 
ensure the existence of an 
$\ep$-independent constant 
$C=C(\|u_{0x}\|_{H^4(\TT;TN)})>1$ 
such that 
\begin{equation}
\frac{1}{C}N_k(u^{\ep}(t))
\leqslant 
\|u_x^{\ep}(t)\|_{H^k(\TT;TN)}
\leqslant 
C\,N_k(u^{\ep}(t))
\quad \
\text{for any}
\quad \
t\in [0,T_{\ep}^{\star}]. 
\label{eq:timecut}
\end{equation}
The equivalence of $N_k(u^{\ep})$
and $\|u_x^{\ep}\|_{H^k(\TT;TN)}$ on 
$[0,T_{\ep}^{\star}]$ will be used frequently below. 
We shall show that there exists  
$T=T(\|u_{0x}\|_{H^4(\TT;TN)})>0$ 
which is independent of $\ep\in (0,1]$ and $k$
such that $T^{\star}_{\ep}\geqslant T$ uniformly in $\ep\in (0,1]$ 
and that $\left\{N_k(u^{\ep})\right\}_{\ep\in (0,1]}$ 
is bounded in $L^{\infty}(0,T)$. 
If it is true, this together with \eqref{eq:timecut} 
implies that $\left\{u_x^{\ep}\right\}_{\ep\in (0,1]}$ is 
bounded in $L^{\infty}(0,T;H^k(\TT;TN))$. 
\par 
Having them in mind, 
we move on to the  
estimate for $N_k(u^{\ep})$.  
We set $u=u^{\ep}$, $V_k=V_k^{\ep}$, 
$\Lambda=\Lambda^{\ep}$, 
$\Lambda_1=\Lambda_1^{\ep}$, $\Lambda_2=\Lambda_2^{\ep}$,
$\|\cdot\|_{H^0(\TT;TN)}=\|\cdot\|_{L^2(\TT;TN)}=\|\cdot\|_{L^2}$, 
$\|\cdot\|_{H^m(\TT;TN)}=\|\cdot\|_{H^m}$ for $m=1,\ldots,k$, 
and 
$|\cdot|_h=\left\{h(\cdot,\cdot)\right\}^{1/2}$  
for ease of notation. 
Since $h$ is a hermitian metric, 
$h(J_uY_1,J_uY_2)=h(Y_1,Y_2)$ holds for any 
$Y_1,Y_2\in\Gamma(u^{-1}TN)$. 
Since $(N,J,h)$ is a K\"ahler manifold, 
$\nabla_xJ_u=J_u\nabla_x$ and $\nabla_tJ_u=J_u\nabla_t$ hold.  
Any positive constant which depends on  
$a$, $b$, $c$, $\lambda$, $k$, 
$\|u_{0x}\|_{H^4}$ 
and not on $\ep\in (0,1]$ will be denoted by the same $C$. 
Note that $k\geqslant 4$ and the Sobolev embedding 
$H^1(\TT)\subset C(\TT)$
yield 
$\|\nabla_x^4u_x\|_{L^{\infty}(0,T_{\ep}^{\star};L^2)}\leqslant C$
and $\|\nabla_x^mu_x\|_{L^{\infty}((0,T_{\ep}^{\star})\times \TT)}\leqslant C$ 
for $m=0,1,\ldots,3$. 
These properties will be used without any comment in this section.
\par 
We now investigate the energy estimate for $\|V_k\|_{L^2}^2$. 
The starting point is the observation that 
\begin{align}
\frac{1}{2}\frac{d}{dt}
\|V_k\|_{L^2}^2
&=
\int_{\TT}
h(\nabla_tV_k,V_k)
dx.
\label{eq:eqV_k}
\end{align}
To evaluate 
the right hand side
(denoted by RHS hereafter), 
we compute $\nabla_tV_k=\nabla_t\nabla_x^ku_x+\nabla_t\Lambda$ below. 
\par 
We begin with the computation of $\nabla_t\nabla_x^ku_x$. 
Recalling that $\nabla_xu_t=\nabla_tu_x$ and 
$(\nabla_x\nabla_t-\nabla_t\nabla_x)Y=R(u_x,u_t)Y$ 
for any $Y\in \Gamma(u^{-1}TN)$, 
we have
\begin{align}
\nabla_t\nabla_x^ku_x
&=
\nabla_x^{k+1}u_t
+
\sum_{m=0}^{k-1}
\nabla_x^{k-1-m}
\left\{
R(u_t,u_x)\nabla_x^mu_x
\right\}
=:
\nabla_x^{k+1}u_t
+Q
.
\label{eq:na1}
\end{align}
We compute the RHS of \eqref{eq:na1} using \eqref{eq:eppde}. 
We first look at the first term of the RHS of 
\eqref{eq:na1}. 
A simple computation shows 
\begin{align}
\nabla_x^{k+1}u_t
&=
-\ep\nabla_x^4(\nabla_x^ku_x)
+a\,J_u\nabla_x^4(\nabla_x^ku_x)
+\lambda\,J_u\nabla_x^2(\nabla_x^ku_x)
+b\,Q_{1,1}+c\,Q_{1,2}, 
\label{eq:nab_1}
\end{align}
where 
$Q_{1,1}
=
\nabla_x^{k+1}\left\{
R(\nabla_xu_x,u_x)J_uu_x
\right\}$ 
and  
$Q_{1,2}
=
\nabla_x^{k+1}\left\{
R(J_uu_x,u_x)\nabla_xu_x
\right\}$.
In the computation of $Q_{1,1}$ and $Q_{1,2}$,  
the assumption $\nabla R=0$  
is used. 
Indeed, by using \eqref{eq:sym} in the previous section 
and 
the product formula for covariant differentiation, 
we deduce 
\begin{align}
Q_{1,1}
&=
\sum_{\substack{q+r+s=k+1, \\ q,r,s\geqslant 0}}
\frac{(k+1)!}{q!r!s!}
\,
R(\nabla_x^{q+1}u_x,\nabla_x^{r}u_x)J_u\nabla_x^{s}u_x
\nonumber
\\
&=
R(\nabla_x^{k+2}u_x,u_x)J_uu_x
+(k+1)R(\nabla_x^{k+1}u_x,\nabla_xu_x)J_uu_x
\nonumber
\\
&\quad 
+(k+1)R(\nabla_x^{k+1}u_x,u_x)J_u\nabla_xu_x
+R(\nabla_xu_x, \nabla_x^{k+1}u_x)J_uu_x
\nonumber
\\
&\quad
+R(\nabla_xu_x, u_x)J_u\nabla_x^{k+1}u_x
\nonumber
\\
&\quad
+
\sum_{\substack{q+r+s=k+1, 
\\ 0\leqslant q \leqslant k-1, 
0\leqslant r,s\leqslant k}}
\frac{(k+1)!}{q!r!s!}
\,
R(\nabla_x^{q+1}u_x,\nabla_x^{r}u_x)J_u\nabla_x^{s}u_x. 
\nonumber
\intertext{Furthermore, the Sobolev embedding and the
 Gagliardo-Nirenberg inequality imply }
Q_{1,1}
&=
R(\nabla_x^{k+2}u_x,u_x)J_uu_x
+(k+1)R(\nabla_x^{k+1}u_x,\nabla_xu_x)J_uu_x
\nonumber
\\
&\quad 
+(k+1)R(\nabla_x^{k+1}u_x,u_x)J_u\nabla_xu_x
+R(\nabla_xu_x, \nabla_x^{k+1}u_x)J_uu_x
\nonumber
\\
&\quad
+R(\nabla_xu_x, u_x)J_u\nabla_x^{k+1}u_x
+
\mathcal{O}
\left(
\sum_{m=0}^k
|\nabla_x^mu_x|_h
\right). 
\label{eq:Q11}
\intertext{In the same way as  we compute $Q_{1,1}$, 
we compute $Q_{1,2}$ to obtain }
Q_{1,2}
&=
\sum_{\substack{q+r+s=k+1, \\ q,r,s\geqslant 0}}
\frac{(k+1)!}{q!r!s!}
\,
R(J_u\nabla_x^{q}u_x,\nabla_x^{r}u_x)\nabla_x^{s+1}u_x 
\nonumber
\\
&=
R(J_uu_x,u_x)\nabla_x^{k+2}u_x
+(k+1)R(J_u\nabla_xu_x,u_x)\nabla_x^{k+1}u_x
\nonumber
\\
&\quad 
+(k+1)R(J_uu_x,\nabla_xu_x)\nabla_x^{k+1}u_x
+
R(J_u\nabla_x^{k+1}u_x,u_x)\nabla_xu_x
\nonumber
\\
&\quad 
+
R(J_uu_x,\nabla_x^{k+1}u_x)\nabla_xu_x
\nonumber
\\
&\quad 
+
\sum_{\substack{q+r+s=k+1, 
\\ 0\leqslant s\leqslant k-1, 0\leqslant q,r\leqslant k}}
\frac{(k+1)!}{q!r!s!}
\,
R(J_u\nabla_x^{q}u_x,\nabla_x^{r}u_x)\nabla_x^{s+1}u_x 
\nonumber
\\
&=
\nabla_x\left\{
R(J_uu_x,u_x)\nabla_x^{k+1}u_x
\right\}
+k\,R(J_u\nabla_xu_x,u_x)\nabla_x^{k+1}u_x
\nonumber
\\
&\quad 
+k\,R(J_uu_x,\nabla_xu_x)\nabla_x^{k+1}u_x
+
R(J_u\nabla_x^{k+1}u_x,u_x)\nabla_xu_x
\nonumber
\\
&\quad 
+
R(J_uu_x,\nabla_x^{k+1}u_x)\nabla_xu_x
+
\mathcal{O}
\left(
\sum_{m=0}^k
|\nabla_x^mu_x|_h
\right).
\label{eq:Q12}
\end{align}
We next look at $Q$ which is the second term of the RHS of 
\eqref{eq:na1}.
By substituting \eqref{eq:eppde},   
\begin{align}
Q&=
-\ep\, 
\sum_{m=0}^{k-1}\nabla_x^{k-1-m}
\left\{
R(\nabla_x^3u_x,u_x)\nabla_x^mu_x
\right\}
\nonumber
\\
&\quad 
+a\,
\sum_{m=0}^{k-1}\nabla_x^{k-1-m}
\left\{
R(J_u\nabla_x^3u_x,u_x)\nabla_x^mu_x
\right\}
\nonumber
\\
&\quad 
+\lambda\,
\sum_{m=0}^{k-1}\nabla_x^{k-1-m}
\left\{
R(J_u\nabla_xu_x,u_x)\nabla_x^mu_x
\right\}
\nonumber
\\
&\quad 
+b\,
\sum_{m=0}^{k-1}\nabla_x^{k-1-m}
\left\{
R(R(\nabla_xu_x,u_x)J_uu_x,u_x)\nabla_x^mu_x
\right\}
\nonumber
\\
&\quad 
+c\,
\sum_{m=0}^{k-1}\nabla_x^{k-1-m}
\left\{
R(R(J_uu_x,u_x)\nabla_xu_x,u_x)\nabla_x^mu_x
\right\}. 
\nonumber
\end{align}
Thus, by using the Sobolev embedding 
and the Gagliardo-Nirenberg inequality, 
we obtain 
\begin{align}
Q&=
\ep\,
\mathcal{O}
(|\nabla_x^{k+2}u_x|_h+|\nabla_x^{k+1}u_x|_h)
+a\,Q_0
+
\mathcal{O}
\left(
\sum_{m=0}^k
|\nabla_x^mu_x|_h
\right), 
\nonumber
\end{align}
where 
$$
Q_0
=
\sum_{m=0}^{k-1}\nabla_x^{k-1-m}
\left\{
R(J_u\nabla_x^3u_x,u_x)\nabla_x^mu_x
\right\}.
$$
In the same way as we compute $Q_{1,1}$ and $Q_{1,2}$, 
we use the product formula for covariant differentiation 
to deduce 
\begin{align}
Q_0&=
\nabla_x^{k-1}
\left\{
R(J_u\nabla_x^3u_x,u_x)u_x
\right\}
+
\nabla_x^{k-2}
\left\{
R(J_u\nabla_x^3u_x,u_x)\nabla_xu_x
\right\}
\nonumber
\\
&\quad 
+
\mathcal{O}
\left(
\sum_{m=0}^k
|\nabla_x^mu_x|_h
\right)
\nonumber
\\
&=
\sum_{\substack{q+r+s=k-1, \\ q,r,s\geqslant 0}}
\frac{(k-1)!}{q!r!s!}
\,
R(J_u\nabla_x^{q+3}u_x,\nabla_x^{r}u_x)\nabla_x^{s}u_x 
\nonumber
\\
&\quad 
+
R(J_u\nabla_x^{k+1}u_x,u_x)\nabla_xu_x
+
\mathcal{O}
\left(
\sum_{m=0}^k
|\nabla_x^mu_x|_h
\right)
\nonumber
\\
&=
R(J_u\nabla_x^{k+2}u_x,u_x)u_x
+
(k-1)R(J_u\nabla_x^{k+1}u_x,\nabla_xu_x)u_x
\nonumber
\\
&\quad 
+
k\,R(J_u\nabla_x^{k+1}u_x,u_x)\nabla_xu_x
+
\mathcal{O}
\left(
\sum_{m=0}^k
|\nabla_x^mu_x|_h
\right).
\nonumber
\end{align} 
Therefore, we have 
\begin{align}
Q&=
\ep\,
\mathcal{O}
(|\nabla_x^{k+2}u_x|_h+|\nabla_x^{k+1}u_x|_h)
+a\,R(J_u\nabla_x^{k+2}u_x,u_x)u_x
\nonumber
\\
&\quad+
a(k-1)R(J_u\nabla_x^{k+1}u_x,\nabla_xu_x)u_x 
+
ak\,R(J_u\nabla_x^{k+1}u_x,u_x)\nabla_xu_x
\nonumber
\\
&\quad 
+
\mathcal{O}
\left(
\sum_{m=0}^k
|\nabla_x^mu_x|_h
\right).  
\label{eq:nab_2}
\end{align}
By collecting \eqref{eq:nab_2} and \eqref{eq:nab_1} with 
\eqref{eq:Q11} and with \eqref{eq:Q12}, we have 
\begin{align}
\nabla_t\nabla_x^ku_x
&=
\left\{
-\ep\nabla_x^4
+a\,J_u\nabla_x^4
+\lambda\,J_u\nabla_x^2
\right\}
\nabla_x^ku_x
+
\ep\,
\mathcal{O}
\left(
|\nabla_x^{k+2}u_x|_h+|\nabla_x^{k+1}u_x|_h
\right)
\nonumber
\\
&\quad 
+a\,R(J_u\nabla_x^{k+2}u_x,u_x)u_x
+b\,R(\nabla_x^{k+2}u_x,u_x)J_uu_x
\nonumber
\\
&\quad
+c\,\nabla_x\left\{
R(J_uu_x,u_x)\nabla_x^{k+1}u_x
\right\}
+b\,R(\nabla_xu_x, u_x)J_u\nabla_x^{k+1}u_x
\nonumber
\\
&\quad+
a(k-1)R(J_u\nabla_x^{k+1}u_x,\nabla_xu_x)u_x 
+
(ak+c)\,R(J_u\nabla_x^{k+1}u_x,u_x)\nabla_xu_x
\nonumber
\\
&\quad 
+
c\,R(J_uu_x,\nabla_x^{k+1}u_x)\nabla_xu_x
+b(k+1)R(\nabla_x^{k+1}u_x,\nabla_xu_x)J_uu_x
\nonumber
\\
&\quad
+b\,R(\nabla_xu_x, \nabla_x^{k+1}u_x)J_uu_x
+b(k+1)R(\nabla_x^{k+1}u_x,u_x)J_u\nabla_xu_x
\nonumber
\\
&\quad 
+ck\,R(J_u\nabla_xu_x,u_x)\nabla_x^{k+1}u_x
+ck\,R(J_uu_x,\nabla_xu_x)\nabla_x^{k+1}u_x
\nonumber
\\
&\quad 
+\mathcal{O}
\left(
\sum_{m=0}^k
|\nabla_x^mu_x|_h
\right).
\label{eq:maya}
\end{align}
Here, 
using (i)-(vi) in the previous section, 
we can rewrite the each term of 
the RHS of \eqref{eq:maya} 
separately. 
The result of the computation is as follows: 
\begin{align}
&R(J_u\nabla_x^{k+2}u_x,u_x)u_x
=
-R(\nabla_x^{k+2}u_x,J_uu_x)u_x,
\quad
(\because \text{(vi)})
\label{eq:rr1}
\\
&R(\nabla_x^{k+2}u_x,u_x)J_uu_x
\nonumber
\\
&\quad=
-R(u_x,J_uu_x)\nabla_x^{k+2}u_x
-R(J_uu_x, \nabla_x^{k+2}u_x)u_x
\quad
(\because \text{(iii)})
\nonumber
\\
&\quad=
R(J_uu_x,u_x)\nabla_x^{k+2}u_x
+R(\nabla_x^{k+2}u_x,J_uu_x)u_x
\quad
(\because \text{(i)})
\nonumber
\\
&\quad=
\nabla_x\left\{
R(J_uu_x,u_x)\nabla_x^{k+1}u_x
\right\}
-R(J_u\nabla_xu_x,u_x)\nabla_x^{k+1}u_x
\nonumber
\\
&\quad \quad
-R(J_uu_x,\nabla_xu_x)\nabla_x^{k+1}u_x
+R(\nabla_x^{k+2}u_x,J_uu_x)u_x
\nonumber
\\
&\quad=
\nabla_x\left\{
R(J_uu_x,u_x)\nabla_x^{k+1}u_x
\right\}
-2\,R(J_u\nabla_xu_x,u_x)\nabla_x^{k+1}u_x
\nonumber
\\
&\quad \quad 
+R(\nabla_x^{k+2}u_x,J_uu_x)u_x, 
\quad
(\because (\text{vi}))
\label{eq:rr2}
\\
&R(J_u\nabla_x^{k+1}u_x,u_x)\nabla_xu_x
=
R(J_uu_x,\nabla_x^{k+1}u_x)\nabla_xu_x 
\quad
(\because (\text{vi}))
\label{eq:rr3}
\\
&\quad=
-R(\nabla_x^{k+1}u_x,\nabla_xu_x)J_uu_x
-R(\nabla_xu_x,J_uu_x)\nabla_x^{k+1}u_x
\quad
(\because \text{(iii)})
\nonumber
\\
&\quad=
-R(\nabla_x^{k+1}u_x,\nabla_xu_x)J_uu_x
+
R(J_u\nabla_xu_x,u_x)\nabla_x^{k+1}u_x,
\quad
(\because \text{(vi)})
\label{eq:rr4}
\\
&R(J_u\nabla_x^{k+1}u_x,\nabla_xu_x)u_x
\nonumber
\\
&\quad=
-R(\nabla_xu_x,u_x)J_u\nabla_x^{k+1}u_x
-R(u_x,J_u\nabla_x^{k+1}u_x)\nabla_xu_x
\quad
(\because \text{(iii)})
\nonumber
\\
&\quad=
-R(\nabla_xu_x,u_x)J_u\nabla_x^{k+1}u_x
+R(J_u\nabla_x^{k+1}u_x,u_x)\nabla_xu_x
\quad
(\because \text{(i)})
\nonumber
\\
&\quad=
-R(\nabla_xu_x,u_x)J_u\nabla_x^{k+1}u_x
-R(\nabla_x^{k+1}u_x,\nabla_xu_x)J_uu_x
\nonumber
\\
&\quad\quad
+
R(J_u\nabla_xu_x,u_x)\nabla_x^{k+1}u_x,
\quad
(\because \eqref{eq:rr4})
\label{eq:rr5}
\\
&R(\nabla_xu_x,\nabla_x^{k+1}u_x)J_uu_x
=
-R(\nabla_x^{k+1}u_x,\nabla_xu_x)J_uu_x, 
\quad
(\because (\text{i}))
\label{eq:rr6}
\\
&R(J_uu_x,\nabla_xu_x)\nabla_x^{k+1}u_x
=
R(J_u\nabla_xu_x,u_x)\nabla_x^{k+1}u_x.
\quad
(\because (\text{vi}))
\label{eq:rr7}
\end{align}
Substituting \eqref{eq:rr1}-\eqref{eq:rr7} into 
\eqref{eq:maya}, 
we obtain 
\begin{align}
\nabla_t\nabla_x^ku_x
&=
\left\{
-\ep\nabla_x^4
+a\,J_u\nabla_x^4
+\lambda\,J_u\nabla_x^2
\right\}
\nabla_x^ku_x
+
\ep\,
\mathcal{O}
\left(
|\nabla_x^{k+2}u_x|_h+|\nabla_x^{k+1}u_x|_h
\right)
\nonumber
\\
&\quad 
+(-a+b)\,R(\nabla_x^{k+2}u_x,J_uu_x)u_x
+(b+c)\,\nabla_x\left\{
R(J_uu_x,u_x)\nabla_x^{k+1}u_x
\right\}
\nonumber
\\
&\quad
+c_1\,R(J_u\nabla_xu_x, u_x)\nabla_x^{k+1}u_x
+c_2\,R(\nabla_xu_x,u_x)J_u\nabla_x^{k+1}u_x
\nonumber
\\
&\quad
+c_3\,R(\nabla_x^{k+1}u_x,\nabla_xu_x)J_uu_x
+c_4\,R(\nabla_x^{k+1}u_x,u_x)J_u\nabla_xu_x
\nonumber
\\
&\quad 
+\mathcal{O}
\left(
\sum_{m=0}^k
|\nabla_x^mu_x|_h
\right), 
\label{eq:maya12}
\end{align}
where 
\begin{align}
c_1&=
-2b+a(k-1)+(ak+c)+c+2ck
=(2k-1)a-2b+(2k+2)c, 
\label{eq:cc1}
\\
c_2&=-(k-1)a+b, 
\label{eq:cc2}
\\
c_3&=
-a(k-1)-(ak+c)-c+b(k+1)-b
=-(2k-1)a+kb-2c,
\label{eq:cc3}
\\
c_4&=(k+1)b.
\label{eq:cc4}
\end{align}
Furthermore, we rewrite  
$R(\nabla_x^{k+1}u_x, \nabla_xu_x)J_uu_x$ 
and 
$R(\nabla_x^{k+1}u_x, u_x)J_u\nabla_xu_x$ 
by using $A_i$ ($i=1,2,3$) defined 
for any $Y\in \Gamma(u^{-1}TN)$
by 
\begin{align}
A_1Y
&=
R(Y,\nabla_xu_x)J_uu_x-R(Y,u_x)J_u\nabla_xu_x,
\nonumber
\\
A_2Y
&=
R(Y,\nabla_xu_x)J_uu_x+
R(Y,J_uu_x)\nabla_xu_x,
\nonumber
\\
A_3Y&=
R(Y,u_x)J_u\nabla_xu_x
+R(Y,J_u\nabla_xu_x)u_x.
\nonumber
\end{align}
\begin{remark}
It is to be emphasized that 
$A_i$($i=1,2,3$) is symmetric, that is, 
\begin{equation}
h(A_iY,Z)=h(Y,A_iZ)
\quad 
(i=1,2,3)
\label{eq:symmm}
\end{equation}
for any $Y,Z\in \Gamma(u^{-1}TN)$. 
If $i=1$, \eqref{eq:symmm} follows from 
\begin{align}
h(A_1Y,Z)
&=
h(R(Y,\nabla_xu_x)J_uu_x,Z)
-h(R(Y,u_x)J_u\nabla_xu_x,Z)
\nonumber
\\
&=
-h(R(\nabla_xu_x,J_uu_x)Y,Z)
-h(R(J_uu_x,Y)\nabla_xu_x,Z)
\nonumber
\\
&\quad
+h(R(u_x,J_u\nabla_xu_x)Y,Z)
+h(R(J_u\nabla_xu_x,Y)u_x,Z)
\quad
(\because (\text{iii}))
\nonumber
\\
&=
-h(R(J_uu_x,Y)\nabla_xu_x,Z)
+h(R(J_u\nabla_xu_x,Y)u_x,Z)
\quad 
(\because (\text{vi}))
\nonumber
\\
&=
h(R(Z,\nabla_xu_x)J_uu_x,Y)
-h(R(Z,u_x)J_u\nabla_xu_x,Y)
\quad
(\because (\text{i}), (\text{ii}))
\nonumber
\\
&=h(A_1Z,Y).
\nonumber
\end{align}
If $i=2$ or $i=3$, 
\eqref{eq:symmm} easily follows from (i) and (ii). 
The property \eqref{eq:symmm} will be used later.
\end{remark}
By using $A_1$, $A_2$, $A_3$, we can write 
\begin{align}
4\,R(Y,\nabla_xu_x)J_uu_x
&=
2A_1Y
+
2\left\{
R(Y,\nabla_xu_x)J_uu_x+R(Y,u_x)J_u\nabla_xu_x
\right\}
\nonumber
\\
&=2A_1Y+A_2Y+A_3Y
\nonumber
\\
&\quad
+R(Y,\nabla_xu_x)J_uu_x-
R(Y,J_uu_x)\nabla_xu_x
\nonumber
\\
&\quad 
+
R(Y,u_x)J_u\nabla_xu_x
-R(Y,J_u\nabla_xu_x)u_x. 
\label{eq:notee1}
\end{align}
Here, it follows that 
\begin{align}
&R(Y,\nabla_xu_x)J_uu_x+R(Y,u_x)J_u\nabla_xu_x
\nonumber
\\
&=
-R(\nabla_xu_x,J_uu_x)Y
-R(J_uu_x,Y)\nabla_xu_x
\nonumber
\\
&\quad
-R(u_x,J_u\nabla_xu_x)Y
-R(J_u\nabla_xu_x,Y)u_x
\quad
(\because (\text{iii}))
\nonumber
\\
&=2\,R(J_u\nabla_xu_x,u_x)Y
+R(Y,J_uu_x)\nabla_xu_x+R(Y,J_u\nabla_xu_x)u_x.
\quad
(\because (\text{i}), (\text{vi}))
\label{eq:notee2}
\end{align} 
Combining \eqref{eq:notee1} and \eqref{eq:notee2}, we have 
\begin{align}
R(Y,\nabla_xu_x)J_uu_x
&=\frac{1}{2}R(J_u\nabla_xu_x,u_x)Y+\frac{1}{2}A_1Y
+\frac{1}{4}A_2Y+\frac{1}{4}A_3Y. 
\label{eq:ppp1}
\end{align}
In the same way, since 
$$
4\,R(Y,u_x)J_u\nabla_xu_x
=
-2A_1Y
+
2\left\{
R(Y,\nabla_xu_x)J_uu_x+R(Y,u_x)J_u\nabla_xu_x
\right\},
$$
we have 
\begin{align}
R(Y,u_x)J_u\nabla_xu_x
&=\frac{1}{2}R(J_u\nabla_xu_x,u_x)Y-\frac{1}{2}A_1Y
+\frac{1}{4}A_2Y+\frac{1}{4}A_3Y. 
\label{eq:ppp2}
\end{align}
Applying \eqref{eq:ppp1} and \eqref{eq:ppp2} with $Y=\nabla_x^{k+1}u_x$
to the RHS of \eqref{eq:maya12}, 
we obtain 
\begin{align}
\nabla_t\nabla_x^ku_x
&=
\left\{
-\ep\nabla_x^4
+a\,J_u\nabla_x^4
+\lambda\,J_u\nabla_x^2
\right\}
\nabla_x^ku_x
+
\ep\,
\mathcal{O}
\left(
|\nabla_x^{k+2}u_x|_h+|\nabla_x^{k+1}u_x|_h
\right)
\nonumber
\\
&\quad 
+d_1\,R(\nabla_x^{k+2}u_x,J_uu_x)u_x
+d_2\,\nabla_x\left\{
R(J_uu_x,u_x)\nabla_x^{k+1}u_x
\right\}
\nonumber
\\
&\quad
+d_3\,R(J_u\nabla_xu_x, u_x)\nabla_x^{k+1}u_x
+d_4\,R(\nabla_xu_x,u_x)J_u\nabla_x^{k+1}u_x
\nonumber
\\
&\quad
+(d_5A_1+d_6A_2+d_7A_3)\nabla_x^{k+1}u_x 
+\mathcal{O}
\left(
\sum_{m=0}^k
|\nabla_x^mu_x|_h
\right), 
\label{eq:ayaya00}
\end{align}
where $d_1=-a+b$, 
$d_2=b+c$, 
$d_3=c_1+c_3/2 +c_4/2$, 
$d_4=c_2$, 
$d_5=(c_3-c_4)/2$, 
$d_6=d_7=(c_3+c_4)/4$, 
and $c_1,\dots,c_4$ are given in 
\eqref{eq:cc1}-\eqref{eq:cc4}. 
For instance, $d_3$ is given by 
\begin{align}
d_3
&=
(2k-1)a-2b+(2k+2)c
+
\frac{1}{2}\left\{
-(2k-1)a+kb-2c
\right\}
+\frac{1}{2}(k+1)b
\nonumber
\\
&=
\left(k-\frac{1}{2}\right)a
+\left(k-\frac{3}{2}\right)b
+(2k+1)c.
\label{eq:d3}
\end{align}
The explicit form of these constants is not required 
except for $d_1$ and $d_3$. 
Furthermore, substituting  
$\nabla_x^ku_x=V_k-\Lambda$ into the RHS of \eqref{eq:ayaya00} 
and noting $\Lambda=\mathcal{O}(|\nabla_x^{k-2}u_x|_h)$, 
we find 
\begin{align}
\nabla_t\nabla_x^ku_x
&=
\ep\,\nabla_x^4\Lambda-a\,J_u\nabla_x^4\Lambda
\nonumber
\\
&\quad 
+
\left\{
-\ep\nabla_x^4
+a\,J_u\nabla_x^4
+\lambda\,J_u\nabla_x^2
\right\}
V_k
+
\ep\,
\mathcal{O}
\left(
|\nabla_x^{k+2}u_x|_h+|\nabla_x^{k+1}u_x|_h
\right)
\nonumber
\\
&\quad 
+d_1\,R(\nabla_x^{2}V_k,J_uu_x)u_x
+d_2\,\nabla_x\left\{
R(J_uu_x,u_x)\nabla_xV_k
\right\}
\nonumber
\\
&\quad
+d_3\,R(J_u\nabla_xu_x, u_x)\nabla_xV_k
+d_4\,R(\nabla_xu_x,u_x)J_u\nabla_xV_k
\nonumber
\\
&\quad
+(d_5A_1+d_6A_2+d_7A_3)\nabla_xV_k 
+\mathcal{O}
\left(
\sum_{m=0}^k
|\nabla_x^mu_x|_h
\right). 
\label{eq:ayaya01}
\end{align}
Here, a simple computation yields  
$$\ep\,\nabla_x^4\Lambda
=\ep\,\mathcal{O}(|\nabla_x^{k+2}u_x|_h+|\nabla_x^{k+1}u_x|_h)
+\mathcal{O}
\left(\sum_{m=0}^k|\nabla_x^{m}u_x|_h\right)
$$
and $\nabla_x^4\Lambda=\nabla_x^4\Phi\nabla_x^{-2}\nabla_x^ku_x$, 
where $\Phi=-(e_1/2a)\Phi_1+(e_2/8a)\Phi_2$.  
Consequently, we derive 
\begin{align}
\nabla_t\nabla_x^ku_x
&=
-a\,J_u\nabla_x^4\Phi\nabla_x^{-2}\nabla_x^ku_x 
+
\left\{
-\ep\nabla_x^4
+a\,J_u\nabla_x^4
+\lambda\,J_u\nabla_x^2
\right\}
V_k
\nonumber
\\
&\quad 
+
\ep\,
\mathcal{O}
\left(
|\nabla_x^{k+2}u_x|_h
+
|\nabla_x^{k+1}u_x|_h
\right)
\nonumber
\\
&\quad 
+d_1\,R(\nabla_x^{2}V_k,J_uu_x)u_x
+d_2\,\nabla_x\left\{
R(J_uu_x,u_x)\nabla_xV_k
\right\}
\nonumber
\\
&\quad
+d_3\,R(J_u\nabla_xu_x, u_x)\nabla_xV_k
+d_4\,R(\nabla_xu_x,u_x)J_u\nabla_xV_k
\nonumber
\\
&\quad
+(d_5A_1+d_6A_2+d_7A_3)\nabla_xV_k 
+\mathcal{O}
\left(
\sum_{m=0}^k
|\nabla_x^mu_x|_h
\right). 
\label{eq:ayaya}
\end{align}
\par  
Next, we compute 
$\nabla_t\Lambda$.
Using the product formula and noting 
that
$\nabla_tu_x=\nabla_xu_t
=\mathcal{O}
\left(
\displaystyle 
\sum_{m=0}^4|\nabla_x^mu_x|_h
\right),
$ 
we have  
\begin{align}
\nabla_t\Lambda
&=
-\frac{e_1}{2a}
R(\nabla_t\nabla_x^{k-2}u_x,u_x)u_x
-\frac{e_1}{2a}
R(\nabla_x^{k-2}u_x,\nabla_tu_x)u_x
\nonumber
\\
&\quad
-\frac{e_1}{2a}
R(\nabla_x^{k-2}u_x,u_x)\nabla_tu_x
+
\frac{e_2}{8a}R(J_uu_x,u_x)J_u\nabla_t\nabla_x^{k-2}u_x
\nonumber
\\
&\quad 
+\frac{e_2}{8a}R(J_u\nabla_tu_x,u_x)J_u\nabla_x^{k-2}u_x
+
\frac{e_2}{8a}R(J_uu_x,\nabla_tu_x)J_u\nabla_x^{k-2}u_x
\nonumber
\\
&=
-\frac{e_1}{2a}
R(\nabla_t\nabla_x^{k-2}u_x,u_x)u_x
+
\frac{e_2}{8a}R(J_uu_x,u_x)J_u\nabla_t\nabla_x^{k-2}u_x
\nonumber
\\
&\quad 
+
\mathcal{O}
\left(
|\nabla_x^{k-2}u_x|_h
\sum_{m=0}^4|\nabla_x^mu_x|_h
\right).
\label{eq:Lambda_t}
\end{align}
By the same computation as that we obtain $\nabla_t\nabla_x^{k}u_x$, 
we find 
\begin{align}
\nabla_t\nabla_x^{k-2}u_x
&=
-\ep\nabla_x^4(\nabla_x^{k-2}u_x)
+a\,J_u\nabla_x^4(\nabla_x^{k-2}u_x)
+
\mathcal{O}
\left(
\sum_{m=0}^k
|\nabla_x^mu_x|_h
\right)
\nonumber
\\
&=
\ep\,
\mathcal{O}
(|\nabla_x^{k+2}u_x|_h)
+a\,J_u\nabla_x^{k+2}u_x
+
\mathcal{O}
\left(
\sum_{m=0}^k
|\nabla_x^mu_x|_h
\right).
\label{eq:k-2}
\end{align}
Substituting \eqref{eq:k-2} into \eqref{eq:Lambda_t} 
and observing  
$a\,J_u\nabla_x^{k+2}u_x=
\nabla_x^{-2}(a\,J_u\nabla_x^4)\nabla_x^ku_x$, 
we obtain 
\begin{align}
\nabla_t\Lambda
&=
-\frac{e_1}{2a}
R(a\,J_u\nabla_x^{k+2}u_x,u_x)u_x
+
\frac{e_2}{8a}R(J_uu_x,u_x)J_ua\,J_u\nabla_x^{k+2}u_x
\nonumber
\\
&\quad 
+
\ep\,
\mathcal{O}
(|\nabla_x^{k+2}u_x|_h)
+
\mathcal{O}
\left(
\sum_{m=0}^k|\nabla_x^mu_x|_h
+
|\nabla_x^{k-2}u_x|_h
\sum_{m=0}^4|\nabla_x^mu_x|_h
\right)
\nonumber
\\
&=
\Phi\nabla_x^{-2}(a\,J_u\nabla_x^4)\nabla_x^{k}u_x 
\nonumber
\\
&\quad 
+
\ep\,
\mathcal{O}
(|\nabla_x^{k+2}u_x|_h)
+
\mathcal{O}
\left(
\sum_{m=0}^k|\nabla_x^mu_x|_h
+
|\nabla_x^{k-2}u_x|_h
\sum_{m=0}^4|\nabla_x^mu_x|_h
\right). 
\label{eq:Lambda_t2}
\end{align}
\par 
Consequently, by combining \eqref{eq:ayaya} and \eqref{eq:Lambda_t2}, 
we derive 
\begin{align}
\nabla_tV_k
&=
-[a\,J_u\nabla_x^4, \Phi\nabla_x^{-2}]\nabla_x^ku_x 
+
\left\{
-\ep\nabla_x^4
+a\,J_u\nabla_x^4
+\lambda\,J_u\nabla_x^2
\right\}
V_k
\nonumber
\\
&\quad 
+
\ep\,
\mathcal{O}
\left(
|\nabla_x^{k+2}u_x|_h
+
|\nabla_x^{k+1}u_x|_h
\right)
\nonumber
\\
&\quad 
+d_1\,R(\nabla_x^{2}V_k,J_uu_x)u_x
+d_2\,\nabla_x\left\{
R(J_uu_x,u_x)\nabla_xV_k
\right\}
\nonumber
\\
&\quad
+d_3\,R(J_u\nabla_xu_x, u_x)\nabla_xV_k
+d_4\,R(\nabla_xu_x,u_x)J_u\nabla_xV_k
\nonumber
\\
&\quad
+(d_5A_1+d_6A_2+d_7A_3)\nabla_xV_k 
\nonumber
\\
&\quad 
+
\mathcal{O}
\left(
\sum_{m=0}^k|\nabla_x^mu_x|_h
+
|\nabla_x^{k-2}u_x|_h
\sum_{m=0}^4|\nabla_x^mu_x|_h
\right),  
\label{eq:yuki}
\end{align}
where the symbol $[\cdot,\cdot]$ denotes the commutator, 
that is,  
$$[a\,J_u\nabla_x^4, \Phi\nabla_x^{-2}]\nabla_x^ku_x
=
a\,J_u\nabla_x^4(\Phi\nabla_x^{-2}\nabla_x^ku_x)
-\Phi\nabla_x^{-2}(a\,J_u\nabla_x^4\nabla_x^ku_x), 
$$
which plays the crucial role in the proof 
and is to be computed below.
We start with 
\begin{align}
[a\,J_u\nabla_x^4, \Phi\nabla_x^{-2}]
&=
-\frac{e_1}{2}
[J_u\nabla_x^4,\Phi_1\nabla_x^{-2}]
+
\frac{e_2}{8}
[J_u\nabla_x^4,\Phi_2\nabla_x^{-2}]. 
\label{eq:Comm}
\end{align}
In what follows, we use $\nabla_x^{-2}$ and $\nabla_x^{-1}$
which does not make sense in general. 
Fortunately, however, 
this makes sense in our computation, 
because they always act on the image of $\nabla_x^2$. 
First, from the product formula and
 $J_u\nabla_x=\nabla_xJ_u$, it follows that 
\begin{align}
[J_u\nabla_x^4,\Phi_1\nabla_x^{-2}]
&=J_u\Phi_1\nabla_x^2
+4J_u(\nabla_x\Phi_1)\nabla_x
+6J_u(\nabla_x^2\Phi_1)
\nonumber
\\
&\quad
+4J_u(\nabla_x^3\Phi_1)\nabla_x^{-1}
+J_u(\nabla_x^4\Phi_1)\nabla_x^{-2}
-\Phi_1J_u\nabla_x^2
\nonumber
\\
&=
-2\Phi_1J_u\nabla_x^2
+(J_u\Phi_1+\Phi_1J_u)\nabla_x^2
+4J_u(\nabla_x\Phi_1)\nabla_x
\nonumber
\\
&\quad 
+6J_u(\nabla_x^2\Phi_1)+4J_u(\nabla_x^3\Phi_1)\nabla_x^{-1}
+J_u(\nabla_x^4\Phi_1)\nabla_x^{-2}.  
\label{eq:comm1}
\end{align}  
Here $(\nabla_x^{\ell}\Phi_1)$, $\ell=1,\ldots,4$, 
are defined by 
$(\nabla_x\Phi_1)Y=\nabla_x(\Phi_1Y)-\Phi_1(\nabla_xY)$, 
and $(\nabla_x^{\ell}\Phi_1)Y
=\nabla_x\left\{(\nabla_x^{\ell-1}\Phi_1)Y\right\}
-(\nabla_x^{\ell-1}\Phi_1)\nabla_xY$, $\ell=2,3,4$, 
for any $Y\in \Gamma(u^{-1}TN)$.  
Recalling the definition of $\Phi_1$ and 
using the properties (i)-(vi), 
\eqref{eq:rr2}, 
\eqref{eq:ppp1}, and \eqref{eq:ppp2}, 
we deduce  
\begin{align}
&-2\Phi_1J_u\nabla_x^2
=
-2R(J_u\nabla_x^2\cdot, u_x)u_x
=
2R(\nabla_x^2\cdot,J_uu_x)u_x, 
\quad
(\because \text{(vi)})
\label{eq:comm11}
\\
&(J_u\Phi_1+\Phi_1J_u)\nabla_x^2
=
J_uR(\nabla_x^2\cdot,u_x)u_x
+
R(J_u\nabla_x^2\cdot,u_x)u_x
\nonumber
\\
&\quad 
=R(\nabla_x^2\cdot,u_x)J_uu_x
-R(\nabla_x^2\cdot, J_uu_x)u_x
\quad 
(\because \text{(iv), (vi)})
\nonumber
\\&\quad 
=
\nabla_x\left\{
R(J_uu_x,u_x)\nabla_x\cdot
\right\}
-2R(J_u\nabla_xu_x,u_x)\nabla_x, 
\quad 
(\because \eqref{eq:rr2})
\label{eq:comm12}
\\
&
4\,J_u(\nabla_x\Phi_1)\nabla_x
=
4\,J_u
\left[
\nabla_x\left\{
R(\nabla_x\cdot,u_x)u_x
\right\}
-R(\nabla_x^2\cdot, u_x)u_x
\right]
\nonumber
\\
&\quad
=
4\left\{
R(\nabla_x\cdot,\nabla_xu_x)J_uu_x
+
R(\nabla_x\cdot,u_x)J_u\nabla_xu_x
\right\}
\quad
(\because \text{(iv)})
\nonumber
\\
&\quad
=
4R(J_u\nabla_xu_x,u_x)\nabla_x
+2A_2\nabla_x
+2A_3\nabla_x.
\quad
(\because \eqref{eq:ppp1}, \eqref{eq:ppp2})
\label{eq:comm13}
\end{align}
Substituting 
\eqref{eq:comm11}-\eqref{eq:comm13} into 
\eqref{eq:comm1}, 
we have 
\begin{align}
&-\frac{e_1}{2}
[J_u\nabla_x^4,\Phi_1\nabla_x^{-2}]
\nonumber
\\
&
=
-e_1\,R(\nabla_x^2\cdot, J_uu_x)u_x
-\frac{e_1}{2}
\nabla_x\left\{
R(J_uu_x,u_x)\nabla_x\cdot
\right\}
-e_1\,R(J_u\nabla_xu_x,u_x)\nabla_x
\nonumber
\\
&\quad
-(e_1\,A_2+e_1\,A_3)\nabla_x
\nonumber
\\
&\quad 
-3e_1\,J_u(\nabla_x^2\Phi_1)
-2e_1\,J_u(\nabla_x^3\Phi_1)\nabla_x^{-1}
-\frac{e_1}{2}J_u(\nabla_x^4\Phi_1)\nabla_x^{-2}.
\label{eq:Comm1}  
\end{align}
Next, we consider the second term of the RHS of  
\eqref{eq:Comm}. 
To begin with, we see  
\begin{align}
[J_u\nabla_x^4,\Phi_2\nabla_x^{-2}]
&=(J_u\Phi_2-\Phi_2J_u)\nabla_x^2
+4J_u(\nabla_x\Phi_2)\nabla_x
+6J_u(\nabla_x^2\Phi_2)
\nonumber
\\
&\quad
+4J_u(\nabla_x^3\Phi_2)\nabla_x^{-1}
+J_u(\nabla_x^4\Phi_2)\nabla_x^{-2}. 
\label{eq:comm2}
\end{align}
In the same way as above, we deduce 
\begin{align}
&(J_u\Phi_2-\Phi_2J_u)\nabla_x^2
=
J_uR(J_uu_x,u_x)J_u\nabla_x^2
-R(J_uu_x,u_x)J_uJ_u\nabla_x^2
\nonumber
\\
&\phantom{(J_u\Phi_2-\Phi_2J_u)\nabla_x^2}
=0,
\quad
(\because \text{(iv)})
\label{eq:comm21}
\\
&4\,J_u(\nabla_x\Phi_2)\nabla_x
=
4J_u
\left\{
\nabla_x(\Phi_2\nabla_x\cdot)
-\Phi_2\nabla_x^2
\right\}
\nonumber
\\
&\phantom{4\,J_u(\nabla_x\Phi_2)\nabla_x}
=
4J_u
\left\{
R(J_u\nabla_xu_x,u_x)J_u\nabla_x
+
R(J_uu_x,\nabla_xu_x)J_u\nabla_x
\right\}
\nonumber
\\
&\phantom{4\,J_u(\nabla_x\Phi_2)\nabla_x}
=
-8\,R(J_u\nabla_xu_x,u_x)\nabla_x.
\quad 
(\because \text{(iv),(vi)})
\label{eq:comm22}
\end{align}
Thus, substituting \eqref{eq:comm21} and \eqref{eq:comm22} 
into \eqref{eq:comm2}, 
we have 
\begin{align}
\frac{e_2}{8}[J_u\nabla_x^4,\Phi_2\nabla_x^{-2}]
&=
-e_2\,R(J_u\nabla_xu_x,u_x)\nabla_x
\nonumber
\\
&\quad 
+\frac{3e_2}{4}J_u(\nabla_x^2\Phi_2)
+\frac{e_2}{2}J_u(\nabla_x^3\Phi_2)\nabla_x^{-1}
+\frac{e_2}{8}J_u(\nabla_x^4\Phi_2)\nabla_x^{-2}. 
\label{eq:Comm2}
\end{align}
Therefore, by collecting \eqref{eq:Comm}, 
\eqref{eq:Comm1}, and \eqref{eq:Comm2}, 
we deduce 
\begin{align}
&[a\,J_u\nabla_x^4, \Phi\nabla_x^{-2}]\nabla_x^ku_x
\nonumber
\\
&=
-e_1\,R(\nabla_x^2\nabla_x^ku_x, J_uu_x)u_x
-\frac{e_1}{2}
\nabla_x\left\{
R(J_uu_x,u_x)\nabla_x\nabla_x^ku_x
\right\}
\nonumber
\\
&\quad
+(-e_1-e_2)\,R(J_u\nabla_xu_x,u_x)\nabla_x\nabla_x^ku_x
-(e_1\,A_2+e_1\,A_3)\nabla_x\nabla_x^ku_x
\nonumber
\\
&\quad 
-3e_1\,J_u(\nabla_x^2\Phi_1)\nabla_x^ku_x
-2e_1\,J_u(\nabla_x^3\Phi_1)\nabla_x^{k-1}u_x
-\frac{e_1}{2}J_u(\nabla_x^4\Phi_1)\nabla_x^{k-2}u_x
\nonumber
\\
&\quad 
+\frac{3e_2}{4}J_u(\nabla_x^2\Phi_2)\nabla_x^ku_x
+\frac{e_2}{2}J_u(\nabla_x^3\Phi_2)\nabla_x^{k-1}u_x
+\frac{e_2}{8}J_u(\nabla_x^4\Phi_2)\nabla_x^{k-2}u_x
\nonumber
\\
&=
-e_1\,R(\nabla_x^2\nabla_x^ku_x, J_uu_x)u_x
-\frac{e_1}{2}
\nabla_x\left\{
R(J_uu_x,u_x)\nabla_x\nabla_x^ku_x
\right\}
\nonumber
\\
&\quad
+(-e_1-e_2)\,R(J_u\nabla_xu_x,u_x)\nabla_x\nabla_x^ku_x
-(e_1\,A_2+e_1\,A_3)\nabla_x\nabla_x^ku_x
\nonumber
\\
&\quad 
+\mathcal{O}\left(\sum_{m=0}^k|\nabla_x^mu_x|_h\right). 
\nonumber
\end{align}
This together with 
$\nabla_x^ku_x=\Lambda+\mathcal{O}(|\nabla_x^{k-2}u_x|_h)$ 
concludes 
\begin{align}
&[a\,J_u\nabla_x^4, \Phi\nabla_x^{-2}]\nabla_x^ku_x
\nonumber
\\
&=
-e_1\,R(\nabla_x^2V_k, J_uu_x)u_x
-\frac{e_1}{2}
\nabla_x\left\{
R(J_uu_x,u_x)\nabla_xV_k
\right\}
\nonumber
\\
&\quad
+(-e_1-e_2)\,R(J_u\nabla_xu_x,u_x)\nabla_xV_k
-(e_1\,A_2+e_1\,A_3)\nabla_xV_k
\nonumber
\\
&\quad 
+\mathcal{O}\left(\sum_{m=0}^k|\nabla_x^mu_x|_h\right). 
\label{eq:Com}
\end{align}
Finally, combining \eqref{eq:yuki} and \eqref{eq:Com}, 
we derive 
\begin{align}
\nabla_tV_k
&=
\left\{
-\ep\nabla_x^4
+a\,J_u\nabla_x^4
+\lambda\,J_u\nabla_x^2
\right\}
V_k
\nonumber
\\
&\quad 
+
\ep\,
\mathcal{O}
\left(
|\nabla_x^{k+2}u_x|_h
+
|\nabla_x^{k+1}u_x|_h
\right)
\nonumber
\\
&\quad 
+(d_1+e_1)\,R(\nabla_x^{2}V_k,J_uu_x)u_x
+\left(d_2+\frac{e_1}{2}\right)\,\nabla_x\left\{
R(J_uu_x,u_x)\nabla_xV_k
\right\}
\nonumber
\\
&\quad
+(d_3+e_1+e_2)\,R(J_u\nabla_xu_x, u_x)\nabla_xV_k
+d_4\,R(\nabla_xu_x,u_x)J_u\nabla_xV_k
\nonumber
\\
&\quad
+\left\{d_5A_1+(d_6+e_1)A_2+(d_7+e_1)A_3\right\}\nabla_xV_k 
\nonumber
\\
&\quad 
+
\mathcal{O}
\left(
\sum_{m=0}^k|\nabla_x^mu_x|_h
+
|\nabla_x^{k-2}u_x|_h
\sum_{m=0}^4|\nabla_x^mu_x|_h
\right).
\label{eq:remi}
\end{align}
\par  
We go back to the estimate for \eqref{eq:eqV_k}. 
Using \eqref{eq:remi}, 
we have 
\begin{align}
&
\frac{1}{2}\frac{d}{dt}
\|V_k\|_{L^2}^2
=
\int_{\TT}
h(\nabla_tV_k, V_k)dx
\nonumber
\\
&=
-\ep
\int_{\TT}
h(\nabla_x^4V_k, V_k)dx
+
\ep\,
\int_{\TT}
h(
\mathcal{O}
\left(
|\nabla_x^{k+2}u_x|_h
+
|\nabla_x^{k+1}u_x|_h
\right), 
V_k)dx
\nonumber
\\
&\quad 
+
\int_{\TT}
h(\left\{a\,J_u\nabla_x^4+\lambda\,J_u\nabla_x^2\right\}V_k, 
V_k)dx
\nonumber
\\
&\quad 
+(d_1+e_1)\,
\int_{\TT}
h(R(\nabla_x^{2}V_k,J_uu_x)u_x, V_k)dx
\nonumber
\\
&\quad
+\left(d_2+\frac{e_1}{2}\right)\,
\int_{\TT}
h(\nabla_x
\left\{
R(J_uu_x,u_x)\nabla_xV_k
\right\}, V_k)dx
\nonumber
\\
&\quad 
+(d_3+e_1+e_2)\,
\int_{\TT}
h(
R(J_u\nabla_xu_x,u_x)\nabla_xV_k, V_k)dx
\nonumber
\\
&\quad
+d_4
\,\int_{\TT}
h(
R(\nabla_xu_x,u_x)J_u\nabla_xV_k, V_k)dx
\nonumber
\\
&\quad 
+
\int_{\TT}
h(\left\{d_5A_1+(d_6+e_1)A_2+(d_7+e_1)A_3\right\}\nabla_xV_k, V_k)dx
\nonumber
\\
&\quad
+\int_{\TT}
h\left(\mathcal{O}
\left(
\sum_{m=0}^k|\nabla_x^mu_x|_h
+
|\nabla_x^{k-2}u_x|_h
\sum_{m=0}^4|\nabla_x^mu_x|_h
\right), V_k
\right)dx. 
\nonumber
\end{align}
We compute each term of the above separately. 
By integrating by parts, we obtain 
\begin{align}
\int_{\TT}
h(J_u\nabla_x^4V_k, V_k)dx
&=
\int_{\TT}
h(J_u\nabla_x^2V_k, \nabla_x^2V_k)dx
=0,   
\nonumber
\\
\int_{\TT}
h(J_u\nabla_x^2V_k, V_k)dx
&=
-
\int_{\TT}
h(J_u\nabla_xV_k, \nabla_xV_k)dx
=0.  
\nonumber
\end{align}
In the same way, we integrate by part using (i) and (ii) 
to show
\begin{align}
&
\int_{\TT}
h(\nabla_x
\left\{
R(J_uu_x,u_x)\nabla_xV_k
\right\}, V_k)dx
\nonumber
\\
&=
-
\int_{\TT}
h(
R(J_uu_x,u_x)\nabla_xV_k, 
\nabla_xV_k)dx
\nonumber
\\
&=
-
\int_{\TT}
h(
R(\nabla_xV_k,\nabla_xV_k)J_uu_x,u_x)dx
\quad
(\because (\text{ii}))
\nonumber
\\
&=0. 
\quad
(\because (\text{i})) 
\nonumber
\end{align}
Since  $k\geqslant 4$, we use the Sobolev embedding, 
the Cauchy-Schwartz inequality, 
and \eqref{eq:timecut} to deduce 
\begin{align}
&\int_{\TT}
h\left(\mathcal{O}
\left(
\sum_{m=0}^k|\nabla_x^mu_x|_h
+
|\nabla_x^{k-2}u_x|_h
\sum_{m=0}^4|\nabla_x^mu_x|_h
\right), V_k
\right)dx
\nonumber
\\
&
\leqslant 
C\|u_x\|_{H^k}\|V_k\|_{L^2}
\nonumber
\\
&
\leqslant 
C\,(N_k(u))^2. 
\nonumber
\end{align}
Using the integration by parts, the Young inequality 
$AB\leqslant A^2/2+B^2/2$ for any $A,B\geqslant 0$, 
$\nabla_x^ku_x=V_k+\mathcal{O}(|\nabla_x^{k-2}u_x|_h)$, 
$\ep\leqslant 1$, 
and \eqref{eq:timecut}, 
we deduce 
\begin{align}
&-\ep
\int_{\TT}
h(\nabla_x^4V_k, V_k)dx
+
\ep\,
\int_{\TT}
h(
\mathcal{O}
\left(
|\nabla_x^{k+2}u_x|_h
+
|\nabla_x^{k+1}u_x|_h
\right), 
V_k)dx
\nonumber
\\
&=
-\ep
\int_{\TT}
h(\nabla_x^2V_k, \nabla_x^2V_k)dx
\nonumber
\\
&\qquad
+
\ep\,
\int_{\TT}
h(
\mathcal{O}
\left(
|\nabla_x^{k+2}u_x|_h
+
|\nabla_x^{k+1}u_x|_h
\right), 
\nabla_x^ku_x
+
\mathcal{O}
\left(
|\nabla_x^{k-2}u_x|_h
\right)
)dx
\nonumber
\\
&\leqslant
-\ep
\|\nabla_x^2V_k\|_{L^2}^2
+
\ep\,C
\|\nabla_x^{k+2}u_x\|_{L^2}
(\|\nabla_x^ku_x\|_{L^2}
+
\|\nabla_x^{k-1}u_x\|_{L^2})
+
C\,\|u_x\|_{H^k}^2
\nonumber
\\
&\leqslant
-\ep
\|\nabla_x^2V_k\|_{L^2}^2
+
\frac{\ep}{2}
\|\nabla_x^{k+2}u_x\|_{L^2}^2
+
\frac{\ep\,C^2}{2}
(\|\nabla_x^ku_x\|_{L^2}
+
\|\nabla_x^{k-1}u_x\|_{L^2})^2
+
C\,\|u_x\|_{H^k}^2
\nonumber
\\
&\leqslant
-\ep
\|\nabla_x^2V_k\|_{L^2}^2
+
\frac{\ep}{2}
\|\nabla_x^2V_k\|_{L^2}^2
+
C\,\|u_x\|_{H^k}^2
\nonumber
\\
&\leqslant 
-
\frac{\ep}{2}
\|\nabla_x^2V_k\|_{L^2}^2
+
C\,(N_k(u))^2.
\nonumber
\end{align}
Note that $R(\nabla_xu_x,u_x)J_u$ is symmetric. 
Indeed, 
\begin{align}
h(R(\nabla_xu_x,u_x)J_uY,Z)
&=
h(R(J_uY,Z)\nabla_xu_x,u_x)
\quad 
(\because (\text{ii}))
\nonumber
\\
&=
h(R(J_uZ,Y)\nabla_xu_x,u_x)
\quad 
(\because (\text{vi}))
\nonumber
\\
&=
h(R(\nabla_xu_x,u_x)J_uZ,Y
\quad 
(\because (\text{ii}))
\nonumber
\\
&=h(Y,R(\nabla_xu_x,u_x)J_uZ) 
\nonumber
\end{align} 
for any $Y,Z\in\Gamma(u^{-1}TN)$. 
Therefore, the integration by parts implies 
\begin{align}
&
\,\int_{\TT}
h(
R(\nabla_xu_x,u_x)J_u\nabla_x^{k+1}u_x, 
\nabla_x^ku_x)dx
\nonumber
\\
&=
-\frac{1}{2}\,\int_{\TT}
h(
R(\nabla_x^2u_x,u_x)J_u\nabla_x^{k}u_x, 
\nabla_x^ku_x)dx
\nonumber
\\
&\quad 
-\frac{1}{2}\,\int_{\TT}
h(
R(\nabla_xu_x,\nabla_xu_x)J_u\nabla_x^{k}u_x, 
\nabla_x^ku_x)dx
\nonumber
\\
&\leqslant 
C\|u_x\|_{H^k}^2
\nonumber
\\
&\leqslant 
C\,(N_k(u))^2.
\nonumber 
\end{align}
As we observed \eqref{eq:symmm}, 
$A_i$, $i=1,2,3$, are symmetric. 
Thus, in the same way as above, 
the integration by parts shows
\begin{align}
&\int_{\TT}
h(A_i\nabla_xV_k, V_k)dx
=
-\frac{1}{2}
\int_{\TT}
h((\nabla_xA_i)V_k, V_k)dx
\leqslant 
C\,(N_k(u))^2 
\nonumber
\end{align}
for each $i=1,2,3$. 
Hence, we have 
$$
\int_{\TT}
h(\left\{d_5A_1+(d_6+e_1)A_2+(d_7+e_1)A_3\right\}\nabla_xV_k, V_k)dx
\leqslant 
C(N_k(u))^2.
$$
Collecting them, 
we derive 
\begin{align}
\frac{1}{2}
\frac{d}{dt}
\|V_k\|_{L^2}^2
&\leqslant 
-\frac{\ep}{2}\|\nabla_x^2V_k\|_{L^2}^2 
+(d_1+e_1)\,
\int_{\TT}
h(
R(\nabla_x^{2}V_k, J_uu_x)u_x, V_k)dx
\nonumber
\\
&\quad
+(d_3+e_1+e_2)
\,\int_{\TT}
h(
R(J_u\nabla_xu_x,u_x)\nabla_xV_k, V_k)dx
+C\,(N_k(u))^2. 
\label{eq:V1}
\end{align}
To cancel the second and the third term of the RHS of above, 
we set $e_1$ and $e_2$ so that 
\begin{align}
e_1&=-d_1=a-b, 
\nonumber
\\
e_2&=-d_3-e_1
=
\left(-k-\frac{1}{2}\right)a
+\left(-k+\frac{5}{2}\right)b
+\left(-2k-1\right)c.
\nonumber
\end{align}
Therefore,   
we derive   
\begin{align}
\frac{1}{2}
\frac{d}{dt}
\|V_k^{\ep}\|_{L^2}^2
&\leqslant 
-\frac{\ep}{2}\|\nabla_x^2V_k^{\ep}\|_{L^2}^2
+C(N_k(u^{\ep}))^2.
\label{eq:VVe}
\end{align}
\par 
Concerning  
the uniform estimate for 
$\left\{N_k(u^{\ep})\right\}_{\ep\in (0,1]}$, 
it remains to consider 
the energy estimate for $\|u_x^{\ep}\|_{H^{k-1}}^2$. 
However, 
by using the integration by parts, the Sobolev embedding, 
and the Cauchy-Schwartz inequality repeatedly, 
it is now easy to show     
\begin{align}
\frac{1}{2}
\frac{d}{dt}
\|u_x^{\ep}\|_{H^{k-1}}^2
&\leqslant
-\frac{\ep}{2}\sum_{m=0}^{k-1}
\|\nabla_x^{m+2}u_x^{\ep}\|_{L^2}^2
+ 
C\,(N_k(u^{\ep}))^2.   
\label{eq:k-1}
\end{align} 
Therefore, from \eqref{eq:VVe} and \eqref{eq:k-1}, 
we conclude that 
there exits a positive constant $C$
depending on $a,b,c,k,\lambda, \|u_{0x}\|_{H^4}$ 
and not on $\ep$ such that 
\begin{align}
\frac{d}{dt}(N_k(u^{\ep}))^2
&=
\frac{d}{dt}\left(
\|u_x^{\ep}\|_{H^{k-1}}^2
+
\|V_k^{\ep}\|_{L^2}^2
\right)
\leqslant 
C(N_k(u^{\ep}))^2
\nonumber
\end{align}
on the time-interval $[0,T_{\ep}^{\star}]$. 
This implies 
$(N_k(u^{\ep}(t)))^2
\leqslant 
(N_k(u_0))^2
e^{Ct}
$
for any $t\in [0,T_{\ep}^{\star}]$. 
Thus, by the definition of $T_{\ep}^{\star}$, 
there holds  
\begin{align}
4(N_4(u_0))^2
&=
(N_4(u^{\ep}(T_{\ep}^{\star})))^2
\leqslant 
(N_4(u_0))^2
e^{C_4T_{\ep}^{\star}} 
\nonumber
\end{align}
with $C_4>0$ which depends on 
 $a,b,c,\lambda, \|u_{0x}\|_{H^4}$ 
and not on $\ep$. 
This shows $e^{C_4T_{\ep}^{\star}}\geqslant 4$ 
and hence 
$T_{\ep}^{\star}\geqslant (\log 4) /C_4$ holds. 
Therefore, if we set $T=(\log 4)/C_4$, 
it follows that 
$T_{\ep}^{\star}\geqslant T$ for any $\ep\in (0,1]$ 
and 
$\left\{
N_k(u^{\ep})
\right\}_{\ep\in (0,1]}$   
is bounded in $L^{\infty}(0,T)$. 
\par
As stated before, this shows that $\left\{u_x\right\}_{\ep\in(0,1]}$ 
is bounded in $L^{\infty}(0,T;H^k(\TT;TN))$. 
Hence the standard compactness argument  
shows the existence of a map $u\in C([0,T]\times \TT;N)$ 
and a subsequence $\left\{u^{\ep(j)}\right\}_{j=1}^{\infty}$ of 
$\left\{u^{\ep}\right\}_{\ep\in (0,1]}$
that satisfy 
\begin{alignat}{3}
&u_x^{\ep(j)}\to u_x
\quad
&\text{in}
\quad 
&C([0,T];H^{k-1}(\TT;TN)), 
\nonumber
\\
&u_x^{\ep(j)}\to u_x
\quad
&\text{in}
\quad 
&L^{\infty}(0,T;H^{k}(\TT;TN)) 
\quad
\text{weakly star}
\nonumber
\end{alignat}
as $j\to \infty$, 
and this $u$ is smooth and solves \eqref{eq:pde}-\eqref{eq:data}. 
\par 
Finally, in the general case where 
$u_0\in C(\TT;N)$ and $u_{0x}\in H^k(\TT;TN)$, 
it suffices to modify the above argument slightly 
by taking a sequence 
$\left\{u_{0}^i\right\}_{i=1}^{\infty}\subset C^{\infty}(\TT;N)$ 
such that  
\begin{equation}
u_{0x}^i
\to
u_{0x}
\quad 
\text{in}
\quad 
H^{k}(\TT;TN)
\label{eq:dense}
\end{equation}
as $i\to \infty$. 
We omit the detail, because the argument of this part 
is the same as that in \cite{onodera4}. 
\end{proof}
{\bf Acknowledgements.} \\
The author would like to thank Hiroyuki Chihara 
for helpful comments in \cite{chihara2}. 
This work has been supported by 
JSPS Grant-in-Aid for Scientific Research (C) 
Grant Number 16K05235.

%

\end{document}